\documentclass[leqno, draft]{amsart}
\usepackage[english]{babel}
\usepackage{amsthm}
\usepackage{color}
\usepackage{amssymb}
\usepackage{enumitem}
\usepackage{amsmath,amsfonts,enumerate}
\usepackage{xparse}

\usepackage[draft]{hyperref}

\numberwithin{equation}{section}

\newtheorem{thm}{Theorem}
\newtheorem{lem}[thm]{Lemma}
\newtheorem{defi}[thm]{Definition}

\newlength{\itemeqBlockIndent}
\newlength{\itemeqIndent}
\NewDocumentCommand{\itemeq}{O{\itemeqIndent} O{1.2ex} m m}{%
  \par\vspace*{#2}\noindent
  \hspace*{\itemeqBlockIndent}
  \makebox[0pt][l]{\textup{(#3)}}%
  \hspace*{#1}$\displaystyle #4$\par
}

\NewDocumentCommand{\itemeqnext}{O{\itemeqIndent} O{.8ex} m}{%
  \vspace*{#2}\noindent
  \hspace*{\itemeqBlockIndent}
  \hspace*{#1}$\displaystyle #3$\par
}

\begin{document}

\title[Fuglede theorem]{Fuglede theorem for symmetric spaces of $\tau$-measurable operators}

\author{Denis Potapov}

\address{School of Mathematics and Statistics, UNSW\\
Kensigston 2052, NSW Australia}

\email{d.potapov@unsw.edu.au}

\author{Fedor Sukochev}

\address{School of Mathematics and Statistics, UNSW\\
Kensigston 2052, NSW Australia}

\email{f.sukochev@unsw.edu.au}

\author{Anna Tomskova}

\address{Inha University in Tashkent\\
Tashkent,  Uzbekistan and V.I. Romanovsky Institute of Mathematics of the Academy of Sciences of
Uzbekistan}

\email{a.tomskova@inha.uz}

\author{Dmitriy~Zanin}

\address{School of Mathematics and Statistics, UNSW\\
Kensigston 2052, NSW Australia}

\email{d.zanin@unsw.edu.au}



\begin{abstract}
We extend the classical Fuglede commutativity theorem to the full scale of symmetrically normed operator ideals.
Our main result provides a complete characterization: a symmetric ideal or symmetric operator space of $\tau$-measurable operators satisfies the Fuglede theorem if and only if its commutative core has non-trivial Boyd indices, or equivalently, if it is an interpolation space in the scale of $L_p$-spaces for $1<p<\infty$.
This criterion subsumes all previously known cases, including Lorentz and Schatten classes.
\end{abstract}
\maketitle

\section{Introduction}

Fuglede's theorem is a classical result in operator theory, stating that if \( A \) and \( B \) are normal operators on a Hilbert space \( H \), and a bounded operator \( T \) satisfies
\[
AT = TB,
\]
then it follows that
\[
A^*T = T B^*.
\]
This result, originally established for the case $A=B$ by Fuglede in 1950 \cite{Fug}, and strengthened by Putnam to the two-operator setting, where both
$A,B$ are normal \cite{Put}, plays a foundational role in understanding the interplay between operator commutation and adjoint structures. In particular, for normal operators, the intertwining relation passes from an operator to its adjoint.

Rewriting the condition in terms of commutators, Fuglede's theorem asserts:
\[
[A, T] = 0 \quad \Rightarrow \quad [A^*, T] = 0.
\]
This naturally raises the question of what can be said when the commutator does not vanish, but is merely small in norm. That is, if \( [A, T] \) is small in some operator ideal norm, can one control \( [A^*, T] \) accordingly?

This leads to the study of \emph{quantitative} versions of the Fuglede-Putnam theorem, with inequalities of the form
\[
\|[A^*,T]\| \leq C \|[A, T]\|,
\]
where \( A \) and \( T \) are bounded operators, and \( C \) is a constant depending on the context. Such estimates have attracted considerable attention (see e.g. \cite{ST, JK, LSZ-book, Farforovskaya1994} and the recent work \cite{KPSS2}). These norm-type commutator estimates have deep connections with perturbation theory, operator ideals, and noncommutative geometry. In perturbation theory, such estimates appear naturally in the study of stability of operator spectra under small perturbations (see e.g. \cite{Kato}). Their relevance to operator ideals is investigated in the context of commutators in normed ideals (see e.g. \cite{KPSS}), {while in noncommutative geometry they are central to Connes’ spectral approach, where commutator bounds play a key role in defining geometric structures on noncommutative spaces (see e.g. \cite{Connes}).} 

Recent developments have extended the Fuglede--Putnam framework to a much broader class of \emph{elementary operators} of the form
$
\Delta(X) = \sum_i A_i X B_i,
$
where the classical argument no longer applies directly. In particular, Shulman and Turowska (see {\color{blue}\cite{ShulmanTurowska2006}}) analyzed this general setting and identified geometric conditions, such as bounded essential dimension of the left coefficient family, under which analogues of the Fuglede--Putnam theorem remain valid {(see also \cite{ShulmanTurowska})}.

In the setting of this paper, let \( H \) be a separable Hilbert space, and let \( {B}(H) \) denote the algebra of all bounded linear operators on \( H \).
A  \emph{(quasi-)Banach symmetric ideal} is, in general, a two-sided (quasi-)Banach ideal \( \mathcal{I} \subset {B}({H}) \) equipped with a unitary invariant (quasi-)norm \( \|\cdot\|_{\mathcal{I}} \) satisfying the ideal property.
A (quasi-)Banach symmetric ideals can also be defined via symmetric norms on sequences by associating to each compact operator its sequence of singular values and requiring the norm to depend only on these values in a symmetric way (see e.g. \cite{KS}).

\begin{defi}
A (quasi-)Banach symmetric ideal $\mathcal{I}$ is \textit{Fuglede} if there is $C>0$ such that
\[
\Vert\lbrack A^{\ast},T]\Vert_{\mathcal{I}}\leq
C\ \Vert\lbrack A,T]\Vert_{\mathcal{I}}%
\]
for all normal operators $A$ and all $T\in B(H)$.
\end{defi}

If some (quasi-)Banach symmetric ideal is Fuglede, this simply means that this enhanced form of the Fuglede-Putnam theorem holds for this ideal. {The term ``Fuglede ideal'' was first introduced in \cite[Definition 3.7]{KissinShulman2005}.}

 It was first shown by G. Weiss in 1978 (\cite{Weiss1978}) that the class of Hilbert-Schmidt operators is Fuglede. Later, Abdessemed and Davies \cite{Abdessemed1989} { and Shulman \cite{Shul1996} extended the result to Schatten classes $\mathcal{S}_p$ for $2<p$ and  $1 < p < 2$ respectively.}
However, certain operator ideals, such as the finite-rank operators, $\mathcal{S}_1$, $\mathcal{S}_\infty$, $B(H)$ and quasi-Banach Schatten classes $\mathcal{S}_p$ for $0 < p < 1$, are not Fuglede~\cite{Weiss1981}.



It was established in \cite[Theorem 4.6]{KPSS2} that any ideal with non-trivial Boyd indices is Fuglede. One of the results of this work is  that we provide a new, more transparent, proof of this result (see Section \ref{sec_main_results}).

In the same paper \cite{KPSS2}, the authors posed the question of whether a Fuglede ideal can have trivial Boyd indices (see \cite[Problem 4.10]{KPSS2}).
In this manuscript, we provide a negative answer to this question (see Theorem~\ref{answer}).

{ The Fuglede ideals were also investigated in the earlier paper \cite{KissinShulman2005}, where it was shown that any separable symmetric normed ideal with non-trivial Boyd indices is a Fuglede ideal (see Corollary 3.8 therein).
However, that analysis did not include certain Lorentz ideals singled out in Problem 3.9 of that paper.
In this work we completely settle this remaining case in Theorem~\ref{answer}. 

}

In the special case of  { separable and coseparable}  ideals, a similar question was recently treated in~\cite{Xia2024}. 
In that work, Fuglede-type commutation inequalities were established for a broad class of symmetric  ideals, including the Lorentz classes by connecting the boundedness of a certain double operator integral with the validity of the Fuglede property. { In addition, the author constructed a number of non-Fuglede ideals, among them certain weighted ideals, including the Macaev ideal.}
Our main result (Theorem~\ref{answer}) gives a complete characterization of when a symmetric ideal is Fuglede: namely, an ideal has the Fuglede property if and only if it has non-trivial Boyd indices. In the special case of concrete ideals considered in~\cite{Xia2024}, Theorem~1.2 there establishes a Fuglede--type estimate under additional structural assumptions, and the subsequent model cases (Theorems~1.3--1.6 in~\cite{Xia2024}) are all covered as direct corollaries of Theorem~\ref{answer}. Thus the present work subsumes the ideal-theoretic part of \cite{Xia2024} and extends it to the general setting.

{ It is worth mentioning that there is also a notion of \emph{Fuglede stable ideals} discussed at length  in \cite{KPSS2} (introduced in \cite[Definition~5.3]{KissinShulman2005} and termed there \emph{weakly Fuglede ideals}), 
which is more delicate than the standard Fuglede property.  
For example, although both ideals $\mathcal{S}_1$ and $\mathcal{S}_\infty$ are not Fuglede, the ideal 
$\mathcal{S}_\infty$ \emph{is}  Fuglede stable (see \cite[Proposition~5.4]{KissinShulman2005}), whereas $\mathcal{S}_1$ is \emph{not} Fuglede stable (see \cite[Corollary~5.8]{KissinShulman2005}). It is also known that every Fuglede ideal is automatically Fuglede stable.  Moreover, in \cite[Corollary~5.7]{KissinShulman2005} the authors provide conditions under which an ideal is  Fuglede stable if and only if it is Fuglede.  
In view of this result, our Theorem~\ref{answer} yields a corresponding criterion for an ideal to be Fuglede stable as well.

}

In addition, in this paper we introduce a notion of Fuglede symmetric operator space (see Definition \ref{def_FugOpSp}) and show  that a symmetric operator space is Fuglede if and only if its commutative core  has non-trivial Boyd indices (see Theorem \ref{MainThm_OpSp}).
Moreover, our approach works for arbitrary symmetric Banach function or sequence spaces (considered in  \cite{KPS} or \cite{BSh}) without any additional properties.

For the convenience of the reader, the main definition and results of the paper are collected in Section~\ref{sec_main_results}. 
However, all technical proofs are postponed to the end of the paper and divided into three parts: 
the proof of sufficiency in Theorem~\ref{MainThm_OpSp} (see Section~\ref{proof of safficiency}), 
the proof of necessity in the case of infinite factors (see Section~\ref{sec_infinite factors}), 
and the proof of necessity in the case of finite factors (see Section~\ref{sec_II_1}).

\section{Preliminaries}

{
\subsection*{Commutative theory}

Let $\mathbb{C}, \mathbb{R}, \mathbb{N}, \mathbb{Z}$ denote the sets of complex numbers, real numbers, natural numbers (without 0), and integers, respectively. 


The space \(\ell_\infty\) is the set of bounded complex sequences \(x = (x_n)_{n=1}^\infty \subset \mathbb{C}\) such that
\[
\ell_\infty := \big\{ x : \|x\|_\infty = \sup_n |x_n| < \infty \big\}.
\]}
A Banach space $E \subset\ell_{\infty}$ is called a \textit{symmetric Banach sequence space} if
\begin{itemize}
\item[(i)] $E$ is a Banach lattice with respect to pointwise order, i.e., if $x \in E$ and $|y| \le |x|,$ then $y \in E$ and $\|y\|_E \le \|x\|_E$;
\item[(ii)] $E$ is \textit{symmetric} (or rearrangement-invariant \cite{NSZ}): for any $x\in E$ and any permutation $\pi$ of $\mathbb{N}$, the sequence $x\circ\pi$ belongs to $E$, and $\|x\circ\pi\|_E = \|x\|_E.$
\end{itemize}
For more details on symmetric function and sequence spaces, we refer to \cite{BSh, KPS, LT}.
{ For \(1 \leq p < \infty\), the classical example of symmetric Banach sequence spaces, the space \(\ell_p\) consists of complex sequences \(x = (x_n)_{n=1}^\infty \subset \mathbb{C}\) such that
\[
\ell_p := \big\{ x : \|x\|_p =\sum_{n=1}^\infty |x_n|^p < \infty \big\}.
\]
For $n \in \mathbb{N}$ and $1 \le p \le \infty$, we denote by $\ell_p^n$ the space $\mathbb{C}^n$ equipped with the norm $\|\cdot\|_p$ defined above.

For each integer \( n \geq 1 \), define the dilation operators \( D_n\) and \( D_{1/n}\)
on \(\ell_\infty \) by setting respectively for $x=(x_1,x_2,\ldots)\in \ell_\infty$,

\[
D_n(x_1, x_2, x_3, \dots) := (\underbrace{x_1, \dots, x_1}_{n \text{ times}}, \underbrace{x_2, \dots, x_2}_{n \text{ times}}, \dots),
\]
\[
D_{1/n}(x_1, x_2, x_3, \dots) := (x_n, x_{2n}, x_{3n}, \dots).
\]
The \textit{Boyd indices} of a symmetric Banach sequence space \( E  \) are defined by \[
\alpha_E = \liminf_{n \geq 2} \frac{\log\|D_{1/n}\|_{E \to E} }{\log(1/n) } = \lim_{n \to \infty} \frac{\log\|D_{1/n}\|_{E \to E} }{\log(1/n) }.
\]
\[
\beta_E = \limsup_{n \geq 2} \frac{\log  \|D_n\|_{E \to E}}{\log n}= \lim_{n \to \infty}\frac{\log  \|D_n\|_{E \to E}}{\log n},
\]
 In general, $0\le \alpha_E\le \beta_E\le 1$. We say that the Boyd indices are \textit{non-trivial} if $0<\alpha_E\le \beta_E<1$.}

Let $\alpha \in (0, \infty]$. A Banach space $E=E(0, \alpha)$ of (equivalence classes of) Lebesgue measurable functions on $(0, \alpha)$ {equipped with the norm $\|\cdot\|_
E$} is called a \textit{symmetric Banach function space} if:

\begin{itemize}
\item[(i)] $E(0, \alpha) \subset L_0(0, \alpha)$ (the space of measurable functions modulo a.e. equality), and $E(0, \alpha)$ is a Banach lattice with respect to the pointwise almost everywhere order and the norm $\|\cdot\|_E$;
\item[(ii)]  if $f \in E(0, \alpha)$ and $g$ is a measurable function on $(0, \alpha)$ with $\mu(g) \le \mu(f)$, then $g \in E(0, \alpha)$ and $\|g\|_E \le \|f\|_E$, where $\mu(f)$ is the non-increasing rearrangement of $f$, that is \[
{\mu(t,f) := \inf\left\{ \lambda > 0 : \left|\left\{ s \in (0, \alpha) : |f(s)| > \lambda \right\}\right| \le t \right\}, \quad t\in (0,\infty).}
\]
\item[(iii)] $E(0,\alpha)\neq\{0\};$
\end{itemize}

{ For \( 1 \leq p \le \infty \), the classical example of symmetric Banach function spaces is
$L_p(0,\infty)$, the set of all {(equivalence classes of) measurable functions} $f \colon (0,\infty) \to \mathbb{C}$ with $\|f\|_p := \left( \int_0^\infty |f(t)|^p \, dt \right)^{1/p} < \infty$, for $1\le p<\infty$ and 
$$ \|f\|_\infty := \operatorname*{ess\,sup}_{t > 0} |f(t)| < \infty,\ p=\infty.$$

For $s>0$ consider
the dilation operator $D_s$ on $L_0(0,\infty)$ defined as
$$(D_s (f) )(t)=f(t/s), \quad t\in (0,\infty).$$

The \textit{Boyd indices} of a symmetric Banach function space \( E(0,\infty) \) are defined by \[
\alpha_E = \liminf_{0 < s < 1} \frac{\log \|D_s\|_{E \to E}}{\log s} = \lim_{s \to 0^+} \frac{\log \|D_s\|_{E \to E}}{\log s}.
\]
\[
\beta_E = \limsup_{s > 1} \frac{\log \|D_s\|_{E \to E}}{\log s} = \lim_{s \to \infty} \frac{\log \|D_s\|_{E \to E}}{\log s},
\]
In general, \( 0 \le \alpha_E \le \beta_E \le 1 \). We say that the Boyd indices are \textit{non-trivial} if \( 0 < \alpha_E \le \beta_E < 1 \).

 Our main results (Theorems~\ref{answer} and~\ref{MainThm_OpSp}) are presented within the (equivalent) framework of interpolation theory rather than in terms of Boyd indices, as the former offers a more precise and widely accepted language among experts in the field. In what follows, we introduce the necessary terminology from interpolation theory and describe its connection to Boyd indices.

 Recall that Banach spaces $(Y,\Vert\cdot\Vert
_{Y})\subset(W,\Vert\cdot\Vert_{W})\subset(X,\Vert\cdot\Vert_{X})$ form an
interpolation triple if any operator $A\in B(X)$ preserving $Y$ and bounded in
$\Vert\cdot\Vert_{Y},$ preserves $W$ and is bounded in $\Vert\cdot\Vert_{W}$. In this case we say that the space $W$ is \textit{an interpolation space} between $Y$ and $X$ and shortly write $W\in{\rm Int}(Y,X)$.

The following result highlights the close connection between the Boyd indices and being an interpolation space between two 
$L_p$ -spaces (see e.g. \cite[Theorem 3.11]{Dirksen2015}, {see also numerous predecessors of this result in a slightly less general setting \cite{Boyd1969, BSh, KPS})}.

\begin{thm}[Boyd's Theorem]\label{thm_Boyd}
Let \( 1 < p < q < \infty \), and let \( E \) be a symmetric Banach sequence space (respectively, symmetric Banach function space). Then \( E \) is an interpolation space with respect to the couple \( (\ell_p, \ell_q) \) (respectively, \( (L_p, L_q) \)) if and only if the Boyd indices of \( E \) satisfy
\[
\frac{1}{q} \le \alpha_E \le \beta_E \le \frac{1}{p}.
\]
\end{thm}

 }

\subsection*{Non-commutative integration theory}

Let $(\mathcal{M}, \tau)$ be a semifinite von Neumann algebra with a faithful normal semifinite trace $\tau$.
A Banach space $E(\mathcal{M}, \tau)$ from $\mathcal{S}(\mathcal{M}, \tau)$, the set of all $\tau$-measurable operators affiliated with $\mathcal{M}$ {(see e.g. \cite[Section 2.3]{DPS-book})},  is called a \textit{symmetric operator space} if there exists a symmetric Banach function space $E$ on $(0, \infty)$ or symmetric Banach sequence space $E$ (depending on the type of $\mathcal M$) such that:
\[
E(\mathcal{M}, \tau) = \{ x \in \mathcal{S}(\mathcal{M}, \tau)\, :\, \mu(x) \in E \}, \quad \|x\|_{E(\mathcal{M}, \tau)} := \|\mu(x)\|_E,
\]
where $\mu(x)$ denotes the generalized singular value function (i.e., the non-increasing rearrangement of $|x|$ with respect to the trace $\tau$). For brevity, throughout the text,  we shall sometimes write $E(\mathcal{M})$ instead of $E(\mathcal{M}, \tau)$, and  $\|x\|_{E}$ instead of $\|x\|_{E(\mathcal{M}, \tau)}$. More details about symmetric operator spaces can be found in \cite{LSZ-book},\cite{DPS-book}, \cite{KS}, \cite{DDdP1989}.

If $\mathcal{M} = L_\infty(0,\alpha)$ and $\tau$ is integration with respect to Lebesgue measure, then $E(\mathcal{M}, \tau)$ coincides with the symmetric Banach function space $E(0,\alpha)$.

If $\mathcal{M} = \ell_\infty$ and $\tau$ is the counting measure $\Sigma$, then $E(\mathcal{M}, \tau)$ is  the symmetric Banach sequence space $E \subset\ell_{\infty}.$

If $\mathcal{M} = B(H)$ (bounded operators on a separable Hilbert space $H$) and $\tau = \operatorname{Tr}$ is the canonical trace, then $E(\mathcal{M}, \tau)$ consists of those $x \in B(H)$ such that the corresponding sequence of singular values $\mu(x)$ belongs to $E.$ This gives rise to \textit{symmetric norm ideals in $B(H)$}, denoted by $\mathcal{E}$.

{
For \( 1 \leq p < \infty \), we set
\[
L_p(\mathcal{M}) = \left\{ x \in \mathcal{S}(\mathcal{M}, \tau) : \tau(|x|^p) < \infty \right\}, \quad \|x\|_p = \left( \tau(|x|^p) \right)^{1/p}.
\]

For $x\in \mathcal{S}(\mathcal{M},\tau)$, the generalized singular value function of $x$ is defined by
$$\mu(t,x)=\inf\left\{s>0: \tau\left(\chi_{(s,\infty)}(|x|)\right)\leq t\right\}, \quad t>0,$$
where $|x|=(x^*x)^{1/2}$ and $\chi_{(s,\infty)}(|x|)$ is the spectral projection of $|x|$ corresponding to the interval $(s,\infty).$
The space \( L_{1,\infty}(\mathcal{M}) \) is defined by
\[
L_{1,\infty}(\mathcal{M}) = \left\{ x \in \mathcal{S}(\mathcal{M}, \tau) : \sup_{t > 0} t\,\mu(t, x) < \infty \right\}.
\]

We equip \( L_{1,\infty}(\mathcal{M}) \) with the quasi-norm
\[
\|x\|_{1,\infty} = \sup_{t > 0} t\,\mu(t, x), \quad x \in L_{1,\infty}(\mathcal{M}).
\]

For $x,y\in \mathcal{S}(\mathcal{M},\tau)$, we have (see e.g. \cite{DPS-book}, \cite{LSZ-book}):
\begin{equation}\label{singular-triangle}
\mu(t+s,x+y)\leq \mu(t,x)+\mu(s,y),\quad s,t>0.
\end{equation}
Then, we have that
\begin{multline*}
\qquad \|x + y\|_{1,\infty} = \sup_{t > 0} t\,\mu(t, x + y) 
\\ \leq \sup_{t > 0} \left( t\,\mu\left( \frac{t}{2}, x \right) + t\,\mu\left( \frac{t}{2}, y \right) \right)
\leq 2 \|x\|_{1,\infty} + 2 \|y\|_{1,\infty}.\qquad 
\end{multline*}
}
The quasi-normed space \( (L_{1,\infty}(\mathcal{M}), \|\cdot\|_{1,\infty}) \) is, in fact, quasi-Banach (see e.g.\ \cite[Section 7]{KS}). Naturally we set
$L_{1,\infty}(0,\infty) = L_{1,\infty}(L_\infty(0,\infty)) 
$

{ The collection of all closed, densely defined operators $x$ in $H$, affiliated with the von Neumann algebra $\mathcal{M}$, and satisfying
\[
\tau\left( e^{|x|}(\lambda, \infty) \right) < \infty \quad \text{for all } \lambda > 0,
\]
will be denoted by $S_0(\mathcal{M},\tau)$.
The elements of $S_0(\mathcal{M},\tau)$ are sometimes called \emph{$\tau$-compact operators}. Note that each $x \in S_0(\mathcal{M},\tau)$ is $\tau$-measurable, that is, $S_0(\mathcal{M},\tau) \subset S(\mathcal{M},\tau)$ (for more details see \cite[\S 2.4]{DPS-book}).

}

A von Neumann algebra $\mathcal{M}$ on a Hilbert space ${H}$ is called a \textit{factor} if its center consists only of scalar multiples of the identity, i.e.,
\[
\mathcal{Z}(\mathcal{M}) = \mathcal{M} \cap \mathcal{M}' = \mathbb{C}\mathbf{1},
\] {where $\mathcal{M}' $ is the commutant of $\mathcal{M} $ and $\mathbf{1}$ is the identity operator on $H$.

 A factor \(\mathcal{M}\) is said to be of \emph{type $\mathrm{I}$} if it contains non-zero abelian projections. In this case it is of type \(\mathrm{I}_n\) if \(\mathcal{M} \cong M_n(\mathbb{C})\), \(n \times n\) complex matrices and of type \(\mathrm{I}_\infty\) if \(\mathcal{M} \cong {B}(H)\) for an infinite-dimensional Hilbert space \(H\).  A factor \(\mathcal{M}\) is said to be of \emph{type $\mathrm{II}$} if it has no minimal projections and admits a faithful normal semifinite trace. In this case it can be of type \(\mathrm{II}_1\) if the trace is finite, e.g., the hyperfinite \(\mathrm{II}_1\) factor, and
     \(\mathrm{II}_\infty\) if the trace is infinite.

A von Neumann algebra is called \emph{atomic} if it contains minimal projections (factors of type $\mathrm{I}$), and \emph{non-atomic} otherwise (factors of type $\mathrm{II}$
). 
If \(\mathcal{M}\) is an atomic factor, then corresponding $E$ in the definition of $E(\mathcal{M}, \tau)$ is a symmetric Banach sequence space.  If \(\mathcal{M}\) is a non-atomic factor, then corresponding $E$ in the definition of $E(\mathcal{M}, \tau)$ is a symmetric Banach function space (can be taken as $E(0,\infty)$). 

In this paper, we focus on factors of types $\mathrm{I}_\infty$,  $\mathrm{II}_1$  and $\mathrm{II}_\infty$.  In particular, we study the case of factors of type $\mathrm{II}_1$ in Section \ref{sec_II_1} and the case of factors of types $\mathrm{I}_\infty$ and $\mathrm{II}_\infty$ in Section \ref{sec_infinite factors}.


A detailed exposition of the theory of von Neumann algebras and symmetric operator spaces can also be found in \cite{DPS-book}.}

{In what follows, we use the notation $c_{abs}$
 to denote an absolute constant, i.e., a positive constant whose value is independent of all parameters and objects under consideration. Whenever the notation $C_{A,B,C}$
 appears, it stands for a constant depending only on the objects $A,B$ and $C$.}

\subsection*{Double operator integrals}

Let \( E \) be a spectral measure on \( \mathbb{R}^2 \) taking values in orthogonal projections in \( \mathcal{M} \). For every pair of measurable subsets \( A, B \subset \mathbb{R}^2 \), the mapping \( E(A) \otimes E(B) \colon L_2( \mathcal{M}) \to L_2( \mathcal{M}) \), given by
\[
E(A) \otimes E(B)(x) = E(A) x E(B), \quad \text{for } x \in L_2( \mathcal{M}),
\]
defines an orthogonal projection on \( L_2( \mathcal{M}) \). The mapping \( G \colon A \times B \mapsto E(A) \otimes E(B) \) can be extended to the algebra of all measurable subsets of \( \mathbb{R}^2 \times \mathbb{R}^2 \) so that the extension is a spectral measure on \( \mathbb{R}^2 \times \mathbb{R}^2 \) with values in orthogonal projections of \( {B}(L_2( \mathcal{M})) \) (see \cite[Definition 2.9 and Remark 3.1]{PSW} and also \cite[Theorem V.2.6]{BS}). We call this measure \( G = E \otimes E \).

For every bounded Borel function \( \varphi \colon \mathbb{R}^2 \times \mathbb{R}^2 \to \mathbb{C} \), the spectral integral
\[
T^{E}_\varphi = \int_{\mathbb{R}^2 \times \mathbb{R}^2} \varphi(u, v) \, dG(u, v), \quad G = E \otimes E 
\]
defines a bounded linear operator \( T^{E}_\varphi\) on \( L_2(\mathcal{M}) \). The norm of \( T^{E}_\varphi \) on \( L_2(\mathcal{M}) \) is controlled by \( \|\varphi\|_\infty \).

The operator \( T^{E}_\varphi \) is called the \textit{double operator integral} of \( \phi \) with respect to the spectral measure \( E \).  {Let $E_1(\mathcal{M},\tau)$ and $E_2(\mathcal{M},\tau)$ be symmetric operator spaces or $E_2(\mathcal{M},\tau)=L_{1,\infty}(\mathcal{M}, \tau)$.}  If \( T^{E}_\varphi\) admits a bounded extension from \( L_2(\mathcal{M}) \cap E_1(\mathcal{M},\tau) \) onto the space \( E_2(\mathcal{M},\tau) \), then we say that \( T^{E}_\varphi \) is bounded from \( E_1(\mathcal{M},\tau) \) to $E_2(\mathcal{M},\tau)$.  {An extensive account of the theory of double operator integrals for $\tau$-measurable operators is  presented in \cite{DDSZ2020}. }


  In the special case when $E$ is supported on $\mathbb{Z}^2$, we have \begin{equation}\label{eq_DOI_discrete}T^E_\varphi(x)=\sum_{i,j\in\mathbb{Z}^2} \varphi(i,j)E(\{i\})x E(\{j\}).
  \end{equation}

If \( \varphi \) is a simple function \( \varphi(u,v) = \sum_{k=1}^n a_k(u)b_k(v) \), then \( T^E_\varphi \) is a multiplication operator with normal coefficients (see e.g. \cite[(2.2)]{KPSS})
\[
T^E_\varphi(x) = \sum_{k=1}^n M_k x N_k,
\]
where
\[
M_k = \int_{\mathbb{R}^2} a_k(u) \, dE(u), \quad N_k = \int_{\mathbb{R}^2} b_k(v) \, dE(v).
\] In particular, if $\varphi(u,v)=(u-v)\chi_A(u)\chi_A(v)$ for a bounded Borel set $A\subset \mathbb{C}$ and $E$ is a spectral measure of a normal operator $y\in \mathcal{S}(\mathcal{M},\tau)$, then 
\begin{multline}
\label{eq_u-v_phi}
    T^{E}_\varphi(x) = \big(\int_{\mathbb{R}^2}u \chi_A(u)dE(u)\big)x \big(\int_{\mathbb{R}^2} \chi_A(v)dE(v)\big)\\-\big(\int_{\mathbb{R}^2}\chi_A(u)dE(u)\big)x \big(\int_{\mathbb{R}^2} v\chi_A(v)dE(v)\big)\\=y E(A) xE(A)-E(A) x E(A) y=[y,E(A)xE(A)].
\end{multline} Similarly, $\psi(u,v)=(\overline{u}-\overline{v})\chi_A(u)\chi_A(v)$, then \begin{multline}
\label{eq_u-v_psi}
    T^{E}_\psi(x) = \big(\int_{\mathbb{R}^2}\overline{u} \chi_A(u)dE(u)\big)x \big(\int_{\mathbb{R}^2} \chi_A(v)dE(v)\big)\\-\big(\int_{\mathbb{R}^2}\chi_A(u)dE(u)\big)x \big(\int_{\mathbb{R}^2} \overline{v}\chi_A(v)dE(v)\big)\\=y^* E(A) xE(A)-E(A) x E(A) y^*=[y^*,E(A)xE(A)].
\end{multline}

The following lemma collects  properties of the map \( \phi \mapsto T^{E}_\varphi \) needed below. The proof directly follows from the spectral theorem.

\begin{lem}
        The mapping \( \phi \mapsto T^{E}_\varphi \) is a homomorphism of the algebra \( \mathcal{B} \) of all bounded Borel functions on \( \mathbb{R}^2 \times \mathbb{R}^2 \) to \( B(L_2(\mathcal{M},\tau)) \). If \( \phi, \psi \in \mathcal{B} \) and \( T^{E}_\varphi, T^{E}_\psi \in B(L_p(\mathcal{M},\tau)) \), then \( T^{E}_{\phi \psi} \in B(L_p(\mathcal{M},\tau)) \) and
\begin{equation}\label{eq_hom}
    T^{E}_{\phi \psi} = T^{E}_\phi T^{E}_\psi.
\end{equation}
\end{lem}





  The following important property of double operator integrals deals with substitution and is one of the key technical tools of our exposition (see \cite[Lemma 2.2]{KPSS} and  \cite[Lemma 8]{PS-2009}
 for the proof).

\begin{lem}\label{lem_substitution}
    Let $\lambda$ be a Borel map on $\mathbb{R}^2$. For a spectral measure $E$ on $\mathbb{R}^2$ define a spectral measure $F$ on $\mathbb{R}^2$ by the rule $$F=E\circ \lambda.$$ If $\varphi$ is a bounded Borel function on $\mathbb{R}^2\times \mathbb{R}^2$, then \[
T^{E\circ \lambda}_\varphi = \int_{\mathbb{R}^2 \times \mathbb{R}^2} \varphi(u, v) \, dG(u, v)=\int_{\mathbb{R}^2 \times \mathbb{R}^2} \varphi(\lambda(u), \lambda(v)) \, dH(u, v)=T^{E}_{\varphi\circ \lambda}, 
\] where $G = F \otimes F$ and $H = E \otimes E.$
\end{lem}

\subsection{Calderón operator and Hilbert transform}
The \textit{Lorentz space} \( \Lambda_{\log}(0,\infty) \) consists of all measurable functions \( f : (0,\infty) \to \mathbb{R} \) such that
\[
\|f\|_{\Lambda_{\log}} := \int_0^\infty \mu(t,f) \, \frac{1}{1 + t} \, dt < \infty.
\] 
For each \( f \in \Lambda_{\log}(0,\infty) \), define the Calderón operator by
\begin{equation}\label{eq_def_Cald}
    (Sf)(t) := \frac{1}{t} \int_0^t f(s)\,ds + \int_t^\infty \frac{f(s)}{s}\,ds.
\end{equation}
This operator maps \( \Lambda_{\log}(0,\infty) \) into \( L_{1,\infty}(0,\infty) + L_\infty(0,\infty) \), and it can be written as the sum of the Hardy-type operators
\[
(Cf)(t) := \frac{1}{t} \int_0^t f(s)\,ds, \qquad (C'f)(t) := \int_t^\infty \frac{f(s)}{s}\,ds.
\]
If $f\in \Lambda_{\log}(\mathbb{R}),$ then the classical Hilbert transform $\mathcal{H}$ is defined by the principal-value integral
\begin{equation}\label{hilbert tr}
(\mathcal{H}f)(t):=p.v.\frac{1}{\pi}\int_{\mathbb{R}}\frac{f(s)}{t-s}ds, \,\  f\in \Lambda_{\log}(\mathbb{R}),
\end{equation}
For more details on these operators in quasi-Banach symmetric spaces we refer the reader to \cite{STZ} and \cite{STZ2}.

\section{Main results}\label{sec_main_results}

\begin{defi}\label{def_FugOpSp}  Let $(\mathcal{M},\tau)$ be a semifinite factor. Let $E$ be a symmetric Banach function space on $(0,\tau(1))$ or symmetric Banach sequence space (depending on the type of $\mathcal{M}$). We say that $E(\mathcal{M},\tau)$ is Fuglede if
$$\|[X^{\ast},Y]\|_{E(\mathcal{M},\tau)}\leq c_E\|[X,Y]\|_{E(\mathcal{M},\tau)}$$
for every normal $X\in E(\mathcal{M},\tau)$ and for every $Y\in\mathcal{M}.$
\end{defi}

\begin{thm}\label{MainThm_OpSp} Let $(\mathcal{M},\tau)$ be a semifinite factor. Let $E$ be a symmetric Banach function space on $(0,\tau(1))$ or symmetric Banach sequence space (depending on the type of $\mathcal{M}$). The space $E(\mathcal{M},\tau)$ is Fuglede if and only if $E\in{\rm Int}(L_p(0,\tau(1)),L_q(0,\tau(1)))$  (respectively, $E\in{\rm Int}(\ell_p,\ell_q)$) for some $1<p<q<\infty.$ 
\end{thm}

\begin{proof}
    Assume that $E\in{\rm Int}(L_p(0,\tau(1)),L_q(0,\tau(1)))$ or $E\in{\rm Int}(\ell_p,\ell_q)$ for some $1<p<q<\infty.$ From Boyd's theorem (Theorem \ref{thm_Boyd}), we have that the Boyd indices of $E$ are non-trivial, that is $\alpha_E>0$ and $\beta_E<1$. In terms of the dilation operator this means that $\|D_s\|_{E\to E}=o(1)$ as $t\to 0$ and $\|D_s\|_{E\to E}=o(t)$ as $t\to \infty$.
    Following \cite[Theorem 6.9]{KPS}, we obtain that {the Hilbert transform $\mathcal{H}$ defined by \eqref{hilbert tr} } is bounded on $E$. This yields that the Calderón operator operator $S$ (see \eqref{eq_def_Cald}) is bounded on $E$ (see \cite[p.140]{KPS}). Finally, applying Lemma \ref{lem_one side}, we conclude that $E(\mathcal{M},\tau)$ is Fuglede.

    On the other hand, assume that $E(\mathcal{M},\tau)$ is Fuglede. If $(\mathcal{M},\tau)$ is a factor of type I$_{\infty}$ or II$_{\infty}$, then applying subsequently Lemma\ref{lem_another side_inf} and Lemma \ref{lem_lower} we have that $\alpha_E>0$. Similarly, applying subsequently Lemma \ref{lem_another side_inf} and Lemma \ref{lem_upper} we obtain that $\beta_E<1$.
   Assume now that $(\mathcal{M},\tau)$ is a factor of II$_{1}$. Then by Lemma \ref{lem_another side_fin}, Lemma \ref{finite alpha lemma} and Lemma \ref{lem_indices_fin} we have that $\alpha_E>0$. Similarly, by Lemma \ref{lem_another side_fin}, Lemma \ref{finite beta lemma} and Lemma \ref{lem_indices_fin}, we obtain that $\beta_E<1$.
 From Boyd's theorem (Theorem \ref{thm_Boyd}), we finally conclude that $E\in{\rm Int}(L_p(0,\tau(1)),L_q(0,\tau(1)))$ or $E\in{\rm Int}(\ell_p,\ell_q)$ for some $1<p<q<\infty.$
\end{proof}

The following result is a consequence of Theorem \ref{MainThm_OpSp} applied to $\mathcal M=B(H)$, which is a factor of type I$_{\infty}$.

\begin{thm}\label{answer} A symmetric norm ideal $\mathcal{E}$ is Fuglede if and only if $E\in{\rm Int}(\ell_p,\ell_q)$ for some $1<p<q<\infty,$ where $E$ is the corresponding  symmetric
Banach sequence space.
\end{thm}

\section{Proof of sufficiency}\label{proof of safficiency}
{In this section, we derive a key estimate (Lemma \ref{lem_one side}) in order to prove that the conditions stated in Theorem \ref{MainThm_OpSp} are sufficient.}

Let $\mathcal{M}$ be a semifinite von Neumann algebra with a faithful normal semifinite trace $\tau$.

\begin{lem} For the function {$\Omega$ defined by
\begin{equation}\label{eq_def_Omega}\Omega(z,w)=\frac{\overline{z}-\overline{w}}{z-w},\quad z\neq w\in\mathbb{C}, \quad \Omega(z,z)=0, \quad z\in\mathbb{C}\end{equation}} and every spectral measure $E$ supported on $\mathbb{Z}^2$ the operator 
$$T^E_{\Omega}:L_1(\mathcal{M},\tau)\to L_{1,\infty}(\mathcal{M},\tau)$$
is bounded
\end{lem}
\begin{proof} Let $p_k=E(\{k\}),$ $k\in\mathbb{Z}^2$ and let $ \{e_k\}_{k \in \mathbb{Z}^2} \subset L_2(\mathbb{T}^2)$ denote the standard orthonormal basis of complex exponentials on the torus, defined by
$$e_k(t){=e^{i\langle k,t\rangle}}=e^{i(k_1 t_1 + k_2 t_2)}, \quad k=(k_1,k_2) \in \mathbb{Z}^2, \; t=(t_1,t_2) \in \mathbb{T}^2.$$
Let $U\in \mathcal{M} \overline{\otimes} L_\infty(\mathbb{T}^2)$ be a unitary operator defined as
$$U=\sum_{k\in\mathbb{Z}^2}p_k\otimes e_k.$$
For every $V\in(L_1\cap L_{\infty})(\mathcal{M},\tau),$ we have
\begin{multline}\label{eq_U}U^{-1}(V\otimes 1)U=(\sum_{k\in\mathbb{Z}^2}p_k\otimes e_{-k})\cdot(V\otimes 1)\cdot(\sum_{\ell\in\mathbb{Z}^2}p_\ell\otimes e_\ell)\\=\sum_{k,\ell\in \mathbb{Z}^2} p_kVp_\ell\otimes e_{\ell-k}.\end{multline}
Define the function  $h:\mathbb{C}\to\mathbb{C}$ by setting
$$h(z)=
\begin{cases}
\frac{\overline{z}}{z},& z\neq0\\
0,&z=0
\end{cases}.$$
Then $h$ is smooth and homogeneous {of degree $0$ on $\mathbb{R}^2\setminus\{0\}$. Note, however, that $h$ is not continuous at the origin. 

Let the \textit{gradient operator} $\nabla_{\mathbb{T}^2}$ be defined as the vector-valued operator $$\nabla_{\mathbb{T}^2}=\frac{1}{i}\left( \frac{\partial}{\partial t_1}, \frac{\partial}{\partial t_2} \right).$$  The components of $\nabla_{\mathbb{T}^2}$, namely, $$D_1=\frac{1}{i} \frac{\partial}{\partial t_1}, \quad  D_2=\frac{1}{i} \frac{\partial}{\partial t_2},$$ are commuting self-adjoint operators, acting on the common core of smooth functions, which is dense in $L_2(\mathbb{T}^2)$.    In terms of the Fourier basis \( \{e_k\}_{k \in \mathbb{Z}^2} \), the operators act as
\[
D_j(e_k) = k_j e_k, \quad j=1,2, \quad\text{for all } k =(k_1,k_2)\in \mathbb{Z}^2,
\] so  each \( e_k \) is a joint eigenfunction of \( D_1 \) and $D_2$ with eigenvalues $k_1$ and $k_2$ respectively.

Using the joint functional calculus for commuting self-adjoint operators $D_1$, $D_2$ we define the operator \( h(\nabla_{\mathbb{T}^2}) := h(D_1,D_2)\) . Then for every $ k=(k_1,k_2) \in \mathbb{Z}^2$,
\[
h(\nabla_{\mathbb{T}^2})(e_k)=h(k_1,k_2)e_k= h(k) e_k.
\]
}
 Let ${\rm id}_{\mathcal{M}}$ be the identity operator on $\mathcal{M}.$ Due to \eqref{eq_U}, we have that
\begin{multline}\label{l6 eq0}
({\rm id}_{\mathcal{M}}\otimes h(\nabla_{\mathbb{T}^2}))(U^{-1}(V\otimes 1)U)=\sum_{k,\ell\in \mathbb{Z}^2} p_kVp_\ell\otimes h(\nabla_{\mathbb{T}^2})e_{\ell-k}\\=\sum_{k,\ell\in\mathbb{Z}^2}p_kVp_\ell\otimes h(\ell-k)e_{\ell-k}=\sum_{k,\ell\in\mathbb{Z}^2} h(\ell-k)p_kVp_\ell\otimes e_{\ell-k}.
\end{multline}

On the other hand, it follows from \eqref{eq_DOI_discrete} and from the equality $\Omega(\ell,k)=h(\ell-k)$ that
$$T^E_{\Omega}(V)=\sum_{k,\ell\in\mathbb{Z}^2}\Omega(\ell,k)p_kVp_\ell=\sum_{k,\ell\in\mathbb{Z}^2}h(\ell-k)p_kVp_\ell.$$
Thus,
\begin{align*}U^{-1}(T^E_{\Omega}(V)\otimes 1)U
&=\Bigl(\sum_{k'\in\mathbb{Z}^2}p_{k'}\otimes e_{-k'}\Bigr)\cdot\Bigl(\sum_{k,\ell\in\mathbb{Z}^2}h(\ell-k)p_kVp_\ell\Bigr)\cdot\Bigl(\sum_{\ell'\in\mathbb{Z}^2}p_{\ell'}\otimes e_{\ell'}\Bigr)\\
&=\sum_{k,\ell,k',\ell'\in\mathbb{Z}^2}h(\ell-k)p_{k'}p_kVp_\ell p_{\ell'}\otimes e_{\ell'-k'}.\end{align*}
Since in the latter sum $p_{k'}p_k=0$ for all  $k'\neq k$  and $p_{\ell}p_{\ell'}=0$ for all $\ell'\neq \ell$, we obtain
\begin{equation}\label{l6 eq1}
U^{-1}(T^E_{\Omega}(V)\otimes 1)U=\sum_{k,\ell\in\mathbb{Z}^2}h(\ell-k)p_kVp_\ell\otimes e_{\ell-k}.
\end{equation}

Combining \eqref{l6 eq0} and \eqref{l6 eq1}, we arrive at
$$U^{-1}(T^E_{\Omega}(V)\otimes 1)U=({\rm id}_{\mathcal{M}}\otimes h(\nabla_{\mathbb{T}^2}))(U^{-1}(V\otimes 1)U).$$
Thus,
\begin{align*}\|T^E_{\Omega}(V)\|_{L_{1,\infty}(\mathcal{M},\tau)}&=\|U^{-1}(T^E_{\Omega}(V)\otimes 1)U\|_{L_{1,\infty}(\mathcal{M}\overline{\otimes}L_{\infty}(\mathbb{T}^2),\tau\otimes m)}\\
&=\|({\rm id}_{\mathcal{M}}\otimes h(\nabla_{\mathbb{T}^2}))(U^{-1}(V\otimes 1)U)\|_{L_{1,\infty}(\mathcal{M}\overline{\otimes}L_{\infty}(\mathbb{T}^2),\tau\otimes m)}\\
&\leq\|{\rm id}_{\mathcal{M}}\otimes h(\nabla_{\mathbb{T}^2})\|_{L_1(\mathcal{M}\overline{\otimes}L_{\infty}(\mathbb{T}^2),\tau\otimes m)\to L_{1,\infty}(\mathcal{M}\overline{\otimes}L_{\infty}(\mathbb{T}^2),\tau\otimes m)}\\
&\times\|U^{-1}(V\otimes 1)U\|_{L_1(\mathcal{M}\overline{\otimes}L_{\infty}(\mathbb{T}^2),\tau\otimes m)}\\
&\leq\|{\rm id}_{\mathcal{M}}\otimes h(\nabla_{\mathbb{T}^2})\|_{L_1(\mathcal{M}\overline{\otimes}L_{\infty}(\mathbb{T}^2))\to L_{1,\infty}(\mathcal{M}\overline{\otimes}L_{\infty}(\mathbb{T}^2)}\|V\|_{L_1(\mathcal{M},\tau)}\end{align*}
for every $V\in (L_1\cap L_{\infty})(\mathcal{M},\tau).$ From \cite[Theorem 2.3]{CSZ} we know that the operator
$${\rm id}_{\mathcal{M}}\otimes h(\nabla_{\mathbb{T}^2}):L_1(\mathcal{M}\overline{\otimes}L_{\infty}(\mathbb{T}^2),\tau\otimes m)\to L_{1,\infty}(\mathcal{M}\overline{\otimes}L_{\infty}(\mathbb{T}^2),\tau\otimes m)$$
is bounded. Hence, $T^E_{\Omega}$ uniquely extends to a bounded operator from $L_1(\mathcal{M},\tau)$ to $L_{1,\infty}(\mathcal{M},\tau).$ 
\end{proof}

{
\begin{lem} For the function $\Omega$ defined in \eqref{eq_def_Omega}, we have that $$T^E_{\Omega}:L_1(\mathcal{M},\tau)\to L_{1,\infty}(\mathcal{M},\tau)$$ is bounded
for every spectral measure $E$ on $\mathbb{C}.$
\end{lem}
\begin{proof} 
Let $E$ be an arbitrary spectral measure on $\mathbb{C}.$ Fix $\varepsilon > 0$ and consider the partition of $\mathbb{C}$ into disjoint $\varepsilon$-squares
\[
\mathbb{C} = \bigcup_{k \in \mathbb{Z}^2} Q_k^\varepsilon, \quad Q_k^\varepsilon := \varepsilon k + \varepsilon [0,1)^2.
\]
We define an approximating discrete spectral measure $E^\varepsilon$ supported on $\varepsilon \mathbb{Z}^2$ by setting
\[
E^\varepsilon(\{\varepsilon k\}) := E(Q_k^\varepsilon), \quad \text{for all } k \in \mathbb{Z}^2.
\]
Since $E^\varepsilon$ is supported on a countable set, by the previous lemma, the double operator integral
\[
T^{E^\varepsilon}_\Omega(x) := \sum_{z,w\in \varepsilon \mathbb{Z}^2} \Omega(z,w) \, E^\varepsilon(\{z\}) \, x \, E^\varepsilon(\{w\})
\]
defines a bounded operator from $L_1(\mathcal{M},\tau)$ to $L_{1,\infty}(\mathcal{M},\tau)$ with a constant { independent of $\varepsilon$}, that is
\[
\|T^{E^\varepsilon}_\Omega(x)\|_{1,\infty} \le C \|x\|_1.
\]

Set
$$\Omega_\varepsilon(z,w)=\Omega(\varepsilon\lfloor\varepsilon^{-1}z\rfloor,\varepsilon\lfloor\varepsilon^{-1}w\rfloor),\quad z,w\in\mathbb{C}.$$
Note that (see e.g. \cite[Lemma 8]{PS-2009})$$T_\Omega^{E^\varepsilon}=T_{\Omega_\varepsilon}^E.$$
We claim that
$$\Omega_{\varepsilon}\to\Omega,\quad \varepsilon\downarrow 0$$ pointwise. 
Indeed, if $z,w\in\mathbb{C}$ are such that $z\neq w,$ then
$$\varepsilon\lfloor\varepsilon^{-1}z\rfloor\to z,\quad \varepsilon\lfloor\varepsilon^{-1}w\rfloor\to w$$
and, therefore,
$$\varepsilon\lfloor\varepsilon^{-1}z\rfloor-\varepsilon\lfloor\varepsilon^{-1}w\rfloor\to z-w\neq0.$$
This immediately yields the claim.

Since the
quasi-norm in $L_{1,\infty}(\mathcal{M},\tau)$ is a Fatou quasi-norm (\cite{LSZ-book}), { it follows that}
\[
\|T^E_\Omega(x)\|_{1,\infty} \le \liminf_{\varepsilon \to 0} \|T^{E^\varepsilon}_\Omega(x)\|_{1,\infty} \le C \|x\|_1, \quad x\in (L_{1}\cap L_2)(\mathcal{M},\tau).
\]
Therefore, $T^E_\Omega$ extends to a bounded operator from $L_1(\mathcal{M},\tau)$ to $L_{1,\infty}(\mathcal{M},\tau)$ for arbitrary spectral measure $E$ on $\mathbb{C}.$
\end{proof}

}

{

\begin{lem}\label{extrapolation lemma}  For the function $\Omega$ defined by \eqref{eq_def_Omega} and the Calderón operator $S$, there is a universal constant $c_{abs}>0$ such that the inequality 
$$\mu(T^E_{\Omega}(V))\leq c_{{\rm abs}}S\mu(V)$$ holds for every spectral measure $E$ on $\mathbb{C}$ and every $V\in \Lambda_{{\rm log}}(\mathcal{M},\tau)$.
\end{lem}


\begin{proof} By the preceding lemma, the operator
$$T^E_{\Omega}:L_1(\mathcal{M},\tau)\to L_{1,\infty}(\mathcal{M},\tau)$$
 is bounded for every spectral measure $E$ on $\mathbb{C}.$ Let ${\rm bar}:\mathbb{C}\to\mathbb{C}$ be the complex conjugation. Note that $\overline{\Omega}=\Omega\circ{\rm bar}$ and, therefore (see Lemma \ref{lem_substitution}), $$T^E_{\overline{\Omega}}=T^{E\circ{\rm bar}}_{\Omega}.$$
Hence, the operator
$$T^E_{\overline{\Omega}}:L_1(\mathcal{M},\tau)\to L_{1,\infty}(\mathcal{M},\tau)$$
is also bounded for every spectral measure $E$ on $\mathbb{C}.$ Thus, the operators
$$T^E_{\Re(\Omega)},T^E_{\Im(\Omega)}:L_1(\mathcal{M},\tau)\to L_{1,\infty}(\mathcal{M},\tau)$$ are bounded
for every spectral measure $E$ on $\mathbb{C}$ as linear combinations of bounded operators. Since $\|\Re(\Omega)\|_\infty,$ $\|\Im(\Omega)\|_\infty<\infty$, it follows that
$$T^E_{\Re(\Omega)},T^E_{\Im(\Omega)}:L_2(\mathcal{M},\tau)\to L_2(\mathcal{M},\tau)$$ are bounded
for every spectral measure $E$ on $\mathbb{C}.$ By interpolation (see \cite[Theorem 4.8]{Dirksen2015}), the operators
$$T^E_{\Re(\Omega)},T^E_{\Im(\Omega)}:L_p(\mathcal{M},\tau)\to L_p(\mathcal{M},\tau),\quad 1<p\leq 2,$$ are bounded
and the norm is controlled by $p'.$ Using Theorem 14 (i)  in \cite{STZ}, we obtain the estimates $$\mu(T^E_{\Re(\Omega)}(V))\leq c_{{\rm abs}}'S\mu(V), \quad \mu(T^E_{\Im(\Omega)}(V))\leq c_{{\rm abs}}''S\mu(V),$$ where $S$ is defined by \eqref{eq_def_Cald}.
Next, due to \cite[Lemma 2.3.12 (d) and Corollary 2.3.16 (a)]{LSZ-book}, we have that 
\begin{multline*}
    \mu(T^E_{\Omega}(V))= \mu(T^E_{\Re(\Omega)+i\Im(\Omega)}(V))\le D_2\mu(T^E_{\Re(\Omega)}(V))+D_2 \mu(T^E_{\Im(\Omega)}(V))\\ \leq (c_{{\rm abs}}'+c_{{\rm abs}}'')D_2S\mu(V)\le   c_{{\rm abs}}S\mu(V),
\end{multline*} where $c_{\rm abs}=c_{{\rm abs}}'+c_{{\rm abs}}''.$
\end{proof}

\begin{lem}\label{lem_one side} If $X\in \mathcal{S}_0(\mathcal{M},\tau)$ is normal and if $Y\in\mathcal{M}$ are such that $[X,Y]\in\Lambda_{{\rm log}}(\mathcal{M},\tau),$ then
$$\mu([X^{\ast},Y])\leq c_{{\rm abs}}S\mu([X,Y]).$$
\end{lem}
\begin{proof} To see this, let $E$ be the spectral measure of $X.$ Fix $n\in\mathbb{N}$ and consider the annulus $A_n=\{\frac1n<|z|<n\}.$ Let $p_n=E(A_n).$ Note that $p_n$ is $\tau$-finite.
	
Applying \eqref{eq_u-v_phi} and \eqref{eq_u-v_psi}, we obtain that 
$$p_n[X,Y]p_n=[X,p_nYp_n]=T^E_{\phi_n}(p_nYp_n),$$
$$p_n[X^{\ast},Y]p_n=[X^{\ast},p_nYp_n]=T^E_{\psi_n}(p_nYp_n),$$
where
$$\phi_n(z,w)=(z-w)\chi_{A_n}(z)\chi_{A_n}(w),\quad \psi_n(z,w)=(\overline{z}-\overline{w})\chi_{A_n}(z)\chi_{A_n}(w),\quad z,w\in\mathbb{C}.$$
It is immediate that $\psi_n=\Omega\cdot\phi_n$, and so by \eqref{eq_hom}, we have $T^E_{\psi_n}=T^E_{\Omega}\circ T^E_{\phi_n}.$
Thus,
$$p_n[X^{\ast},Y]p_n=T^E_{\Omega}(T^E_{\phi_n}(p_n Yp_n))={ T^E_{\Omega}(p_n[X,Y]p_n)}.$$
By Lemma \ref{extrapolation lemma}, we obtain the estimate 
$$\mu(p_n[X^{\ast},Y]p_n)\leq c_{{\rm abs}}S\mu(p_n[X,Y]p_n)\leq c_{{\rm abs}}S\mu([X,Y]).$$
Since $[X^{\ast},Y]\in \mathcal{S}(\mathcal{M},\tau)$ and since $p_n\uparrow {\rm id}_{\mathcal{M}}$ as $n\to\infty,$ { the assertion follows}.
\end{proof}

}

\section{Proof of necessity in the case of factors of types I$_{\infty}$ and II$_{\infty}$}\label{sec_infinite factors}

In this section we prove that the conditions stated in Theorem \ref{MainThm_OpSp} are necessary in the case of factors of types I$_{\infty}$ and II$_{\infty}$. { Through the paper $E_{k,k}\in M_n(\mathbb{C})$, $1\le k\le n,$, $n\in\mathbb{N}$, is the standard matrix unit with a $1$ in the $(k,k)$-entry and $0$ elsewhere and for $A \in M_n(\mathbb{C})$ we set $${\rm diag}(A)=\sum_{k=1}^nE_{k,k}AE_{k,k}.$$}

We need the following technical result, which can be found in \cite[Lemma 6.1]{DdPS2016}. For the reader’s convenience, we include the proof.

\begin{lem}\label{peter lemma} Let $(\mathcal{M},\tau)$ be a semifinite von Neumann algebra with a faithful normal semifinite trace $\tau$. Let $E$ be a symmetric function or sequence space. Then we have
$$\|{\rm diag}(A)\otimes x\|_{E(B(\ell_2)\overline{\otimes}\mathcal{M},{\rm Tr}\otimes\tau)}\leq \|A\otimes x\|_{E(B(\ell_2)\overline{\otimes}\mathcal{M},{\rm Tr}\otimes\tau)}$$
for every $x\in E(\mathcal{M},\tau)$ and for every $A\in B(\ell_2)$ having only finitely many non-zero matrix elements.
\end{lem}
\begin{proof} Let $A=(a_{k,\ell})_{k,\ell\geq1}.$ Fix $n\in\mathbb{N}$ such that $a_{k,\ell}=0$ if either $k>n$ or $\ell>n.$ 
	
For every $\varepsilon=\{\varepsilon_1,\cdots,\varepsilon_n\}\in\{-1,1\}^n,$ we set
$$U_{\varepsilon}=\sum_{k=1}^n\varepsilon_kE_{k,k}.$$ 
We have
\begin{align*}
    \sum_{\varepsilon\in \{-1,1\}^n}U_{\varepsilon}AU_{\varepsilon}&=\sum_{\varepsilon\in \{-1,1\}^n}\Big(\sum_{1\leq k,\ell\leq n}\varepsilon_k\varepsilon_\ell E_{k,k}AE_{\ell,\ell}\Big)\\
&=\sum_{1\leq k,\ell\leq n}\Big(\sum_{\varepsilon\in \{-1,1\}^n}\varepsilon_k\varepsilon_\ell\Big)E_{k,k}AE_{\ell,\ell}.\end{align*}
For $k\neq \ell,$ we have
\begin{equation*}
    \sum_{\varepsilon\in \{-1,1\}^n}\varepsilon_k\varepsilon_\ell=0.
\end{equation*}
For $k=\ell,$ we have
\begin{equation*}\sum_{\varepsilon\in \{-1,1\}^n}\varepsilon_k\varepsilon_\ell=2^n.\end{equation*}
Thus,
\begin{equation*}\sum_{\varepsilon\in \{-1,1\}^n}U_{\varepsilon}AU_{\varepsilon}=2^n\sum_{k=1}^nE_{k,k}AE_{k,k}=2^n{\rm diag}(A).\end{equation*}
Finally, since $U_\varepsilon$ is unitary, it follows that
\begin{align*}
    \|{\rm diag}(A)\otimes x\|_{E(B(\ell_2)\overline{\otimes}\mathcal{M},{\rm Tr}\otimes\tau)}&=2^{-n}\Big\|\Big(\sum_{\varepsilon\in \{-1,1\}^n}U_{\varepsilon}AU_{\varepsilon}\Big)\otimes x\Big\|_{E(B(\ell_2)\overline{\otimes}\mathcal{M},{\rm Tr}\otimes\tau)} \\ & \leq 2^{-n}\sum_{\varepsilon\in \{-1,1\}^n}\|U_{\varepsilon}AU_{\varepsilon}\otimes x\|_{E(B(\ell_2)\overline{\otimes}\mathcal{M},{\rm Tr}\otimes\tau)}\\ &=2^{-n} \cdot 2^n \cdot\|A\otimes x\|_{E(B(\ell_2)\overline{\otimes}\mathcal{M},{\rm Tr}\otimes\tau)}\\ &=\|A\otimes x\|_{E(B(\ell_2)\overline{\otimes}\mathcal{M},{\rm Tr}\otimes\tau)}, 
\end{align*} which completes the proof.
\end{proof}

\begin{defi}\label{def_i_n} Define the linear mapping
$$\iota_n:M_n(\mathbb{C})\to M_{n+1}(\mathbb{C})\otimes M_{n+1}(\mathbb{C})$$
by the formula
$$\iota_n\Big(\sum_{k,\ell=1}^na_{k,\ell}E_{k,\ell}\Big)=\sum_{k,\ell=1}^na_{k,\ell}E_{1,\ell+1}\otimes E_{k+1,1},\quad A=(a_{k,\ell})_{k,\ell=1}^n\in M_n(\mathbb{C}).$$
\end{defi}

\begin{lem}\label{iotan lemma} We have
$$|\iota_n(A)|=U|A|U^{\ast}\otimes E_{1,1},\quad A\in M_n(\mathbb{C}).$$
Here, $U=\sum_{k=1}^nE_{k+1,k}.$
\end{lem}
\begin{proof} Indeed, letting $A=(a_{k,\ell})_{k,\ell=1}^n\in M_n(\mathbb{C})$, we have
\begin{equation}\label{eq_i_n_A}
\begin{aligned}
|\iota_n(A)|^2&=\Big(\sum_{k_1,\ell_1=1}^n\overline{a_{k_1,\ell_1}}E_{\ell_1+1,1}\otimes E_{1,k_1+1}\Big)\\& \qquad\qquad \times\Big(\sum_{k_2,\ell_2=1}^na_{k_2,\ell_2}E_{1,\ell_2+1}\otimes E_{k_2+1,1}\Big)\\ &=\sum_{k_1,\ell_1,k_2,\ell_2=1}^n\overline{a_{k_1,\ell_1}}a_{k_2,\ell_2}(E_{\ell_1+1,1}E_{1,\ell_2+1}\otimes E_{1,k_1+1}E_{k_2+1,1})\\ &=\sum_{k,\ell_1,\ell_2=1}^n\overline{a_{k,\ell_1}}a_{k,\ell_2}E_{\ell_1+1,\ell_2+1}\otimes E_{1,1}\\ &=\sum_{\ell_1,\ell_2=1}^n(A^{\ast}A)_{\ell_1,\ell_2}E_{\ell_1+1,\ell_2+1}\otimes E_{1,1}=U|A|^2U^{\ast}\otimes E_{1,1},
\end{aligned}
\end{equation}
where the latter equality is due to \begin{align*}U A U^*&=\Big(\sum_{\ell_1=1}^nE_{\ell_1+1,\ell_1}\Big)\cdot \Big(\sum_{k,\ell=1}^n a_{k,\ell}E_{k,\ell}\Big)\cdot \Big(\sum_{\ell_2=1}^nE_{\ell_2,\ell_2+1}\Big)\\&= \sum_{k,\ell,\ell_1,\ell_2=1}^n a_{k,\ell}E_{\ell_1+1,\ell_1}E_{k,\ell}E_{\ell_2,\ell_2+1}=\sum_{\ell,\ell_1,\ell_2=1}^n a_{\ell_1,\ell}E_{\ell_1+1,\ell}E_{\ell_2,\ell_2+1}\\&=\sum_{\ell_1,\ell_2=1}^n a_{\ell_1,\ell_2}E_{\ell_1+1,\ell_2+1}.\end{align*} Finally, since $U$ is a { partial isometry such that $U^*U=I_n$, it follows that $$(U|A|U^{\ast})^2=U|A|U^{\ast}U|A|U^{\ast}=U|A|^2U^{\ast}.$$ Using also that $E_{1,1}^2=E_{1,1}$ and taking square root of both sides of \eqref{eq_i_n_A},} we obtain the asserted identity.
\end{proof}

Throughout the rest of this section, $T$ denotes  the \textit{antisymmetric triangular truncation operator} (defined on finite complex matrices), that is,
$$T\Big(\sum_{k,\ell\geq1}a_{k,\ell}E_{k,\ell}\Big)=\sum_{k,\ell\geq1}{\rm sgn}(k-\ell)a_{k,\ell}E_{k,\ell}.$$

\begin{lem}\label{commutator limit lemma}Let $n\ge 1$. For every $A\in M_n(\mathbb{C}),$ there exist sequences $\{A^n_m\}_{m\geq1},$ $\{B_m^n\}_{m\geq1}\subset M_{n+1}(\mathbb{C})\otimes M_{n+1}(\mathbb{C})$ such that
\begin{itemize}
\item $A^n_m$ is normal for every $m\geq1;$
\item $[A^n_m,B^n_m]=\iota_n(A-{\rm diag}(A))$ for every $m\geq 1;$
\item $[(A^n_m)^{\ast},B^n_m]\to \iota_n(T(A))$ as $m\to\infty.$
\end{itemize}	
\end{lem}
\begin{proof} For fixed $n\ge 1$ $m\ge 1$, we set
$$A_m^n=\sum_{1\leq k,\ell\leq n+1}(m^k+im^\ell)E_{k,k}\otimes E_{\ell,\ell},$$
$$B_m^n=\sum_{\substack{2\leq k,\ell\leq n+1\\ k\neq \ell}}\frac1{(m+im^\ell-m^k-im)}a_{\ell-1,k-1}E_{1,k}\otimes E_{\ell,1}.$$
Since $A_m^n$ acts as scalar multiplication on each element of the standard basis  $\{e_k\otimes e_\ell\}$, it is diagonalizable with respect to this basis and hence normal.
	
In order to prove $(ii)$ and $(iii)$, we first find the products
\begin{align*}A_m^nB_m^n&=\sum_{\substack{2\leq k,\ell\leq n+1\\ k\neq \ell\\ 1\leq k',\ell'\leq n+1}}\frac{m^{k'}+im^{\ell'}}{m+im^\ell-m^k-im}a_{\ell-1,k-1}E_{k',k'}E_{1,k}\otimes E_{\ell',\ell'}E_{\ell,1}\\ &
=\sum_{\substack{2\leq k,\ell\leq n+1\\ k\neq \ell}}\frac{m+im^\ell}{m+im^\ell-m^k-im}a_{\ell-1,k-1}E_{1,k}\otimes E_{\ell,1},\end{align*}
\begin{align*}B_m^nA_m^n&=\sum_{\substack{2\leq k,\ell\leq n+1\\ k\neq \ell\\ 1\leq k',\ell'\leq n+1}}\frac{m^{k'}+im^{\ell'}}{m+im^\ell-m^k-im}a_{\ell-1,k-1}E_{1,k}E_{k',k'}\otimes E_{\ell,1}E_{\ell',\ell'}\\ &=\sum_{\substack{2\leq k,\ell\leq n+1\\ k\neq \ell}}\frac{m^k+im}{m+im^\ell-m^k-im}a_{\ell-1,k-1}E_{1,k}\otimes E_{\ell,1},\end{align*}
\begin{align*}(A_m^n)^{\ast}B_m^n&=\sum_{\substack{2\leq k,\ell\leq n+1\\ k\neq \ell\\ 1\leq k',\ell'\leq n+1}}\frac{m^{k'}-im^{\ell'}}{m+im^\ell-m^k-im}a_{\ell-1,k-1}E_{k',k'}E_{1,k}\otimes E_{\ell',\ell'}E_{\ell,1}\\&
=\sum_{\substack{2\leq k,\ell\leq n+1\\ k\neq \ell}}\frac{m-im^\ell}{m+im^\ell-m^k-im}a_{\ell-1,k-1}E_{1,k}\otimes E_{\ell,1},\end{align*}
\begin{align*}B_m^n(A_m^n)^{\ast}&=\sum_{\substack{2\leq k,\ell\leq n+1\\ k\neq \ell\\ 1\leq k',\ell'\leq n+1}}\frac{m^{k'}-im^{\ell'}}{m+im^\ell-m^k-im}a_{\ell-1,k-1}E_{1,k}E_{k',k'}\otimes E_{\ell,1}E_{\ell',\ell'}\\ &=\sum_{\substack{2\leq k,\ell\leq n+1\\ k\neq \ell}}\frac{m^k-im}{m+im^\ell-m^k-im}a_{\ell-1,k-1}E_{1,k}\otimes E_{\ell,1}.\end{align*}
Thus, we find that {(for $\iota_n$ see Definition \ref{def_i_n})}
$$[A_m^n,B_m^n]=\sum_{\substack{2\leq k,\ell\leq n+1\\ k\neq \ell}}a_{\ell-1,k-1}E_{1,k}\otimes E_{\ell,1}=\iota_n(A-{\rm diag}(A)),$$
$$[(A_m^n)^{\ast},B_m^n]=\sum_{\substack{2\leq k,\ell\leq n+1\\ k\neq \ell}}\frac{m-im^\ell-m^k+im}{m+im^\ell-m^k-im}a_{\ell-1,k-1}E_{1,k}\otimes E_{\ell,1}.$$	
Passing $m\to\infty,$ and taking into account that $$\frac{m-im^\ell-m^k+im}{m+im^\ell-m^k-im}\to{\rm sgn}(k-\ell),\quad m\to\infty,\quad 2\leq k,\ell\leq n+1,\quad k\neq \ell,$$
we obtain that
$$[(A_m^n)^{\ast},B_m^n]\to\sum_{\substack{2\leq k,\ell\leq n+1\\ k\neq \ell}}{\rm sgn}(k-\ell)a_{\ell-1,k-1}E_{1,k}\otimes E_{\ell,1}=\iota_n(T(A)),$$ which completes the proof of the lemma.
\end{proof}

\begin{lem}\label{lem_another side_inf} Let $(\mathcal{M},\tau)$ be a semifinite infinite factor. Let $E$ be a symmetric function or sequence space (depending on whether $\mathcal{M}$ is II$_{\infty}$ or I$_{\infty}$). If
$$\|[X^{\ast},Y]\|_{E(\mathcal{M},\tau)}\leq c_{E,\mathcal{M},\tau}\|[X,Y]\|_{E(\mathcal{M},\tau)}$$
for every normal $X\in E(\mathcal{M},\tau)$ and for every $Y\in \mathcal{M},$ then
$$\|T(A)\otimes x\|_{E(B(\ell_2)\overline{\otimes}\mathcal{M},{\rm Tr}\otimes\tau)}\leq 2c_{E,\mathcal{M},\tau}\|A\otimes x\|_{E(B(\ell_2)\overline{\otimes}\mathcal{M},{\rm Tr}\otimes\tau)}$$
for every $x\in E(\mathcal{M},\tau),$ for every $A\in M_n(\mathbb{C})$ and for every $n\in\mathbb{N}.$
\end{lem}
\begin{proof} We will use the von Neumann {factor } $(\mathcal{M}_n,\tau_n)$ with a faithful normal semifinite trace $\tau_n$ defined as
$$(\mathcal{M}_n,\tau_n)=(M_{n+1}(\mathbb{C})\otimes M_{n+1}(\mathbb{C})\otimes\mathcal{M},{\rm Tr}\otimes{\rm Tr}\otimes\tau).$$
Since  $(\mathcal{M}_n,\tau_n)$ and $(\mathcal{M},\tau)$ are {  trace-scalingly $*$-isomorphic (with the scale $(n+1)^2$) and so $(\mathcal{M}_n,\tau_n)$ is also a semifinite infinite factor (see e.g. \cite[Corollary 11.2.17 and Table 11.1]{K-R_II}}, it follows from the assumption of the lemma that
$$\|[X^{\ast},Y]\|_{E(\mathcal{M}_n,\tau_n)}\leq c_{E,\mathcal{M},\tau}\|[X,Y]\|_{E(\mathcal{M}_n,\tau_n)}$$
for every normal $X\in E(\mathcal{M}_n,\tau_n)$ and for every $Y\in\mathcal{M}_n.$

Let $x\in E(\mathcal{M},\tau).$ Substituting $X=A_m^n\otimes |x|$ and $Y=B_m^n\otimes {\rm id}_{\mathcal M},$ where $A_m^n$ and $B_m^n$ are constructed in Lemma \ref{commutator limit lemma}, we obtain $$[X,Y]=[A_m^n\otimes |x|,B_m^n\otimes {\rm id}_{\mathcal M}]=[A_m^n,B_m^n]\otimes |x|,$$$$[X^{\ast},Y]=[(A_m^n)^*\otimes |x|,B_m^n\otimes {\rm id}_{\mathcal M}]=[(A_m^n)^{\ast},B_m^n]\otimes |x|$$ and so
\begin{equation}\label{eq_A_mnB_mn}
    \|[(A_m^n)^{\ast},B_m^n]\otimes |x|\|_{E(\mathcal{M}_n,\tau_n)}\leq c_{E,\mathcal{M},\tau}\|[A_m^n,B_m^n]\otimes |x|\|_{E(\mathcal{M}_n,\tau_n)}.
\end{equation}
{  Since $M_{n+1}(\mathbb{C})\otimes M_{n+1}(\mathbb{C})$ is a finite-dimensional sub-algebra of $\mathcal{M}_n$, the mapping $A \mapsto A \otimes x$ is continuous from $M_{n+1}(\mathbb{C})\otimes M_{n+1}(\mathbb{C})$ to $E(\mathcal{M}_n,\tau_n)$, hence by Lemma \ref{commutator limit lemma}, $[(A_m^n)^{\ast},B_m^n]\otimes |x|  \to \iota_n(T(A))\otimes |x|$ and  $[A_m^n,B_m^n]\otimes |x|\to \iota_n(A-{\rm diag}(A))\otimes |x|$ as $m\to\infty$ in the norm of $E(\mathcal{M}_n,\tau_n)$.}
Taking limit in \eqref{eq_A_mnB_mn} in the norm of  $E(\mathcal{M}_n,\tau_n)$ as  $m\to\infty,$ we infer \begin{equation}\label{eq_estimate}
    \|\iota_n(T(A))\otimes |x|\|_{E(\mathcal{M}_n,\tau_n)}\leq c_{E,\mathcal{M},\tau}\|\iota_n(A-{\rm diag}(A))\otimes |x|\|_{E(\mathcal{M}_n,\tau_n)}.
\end{equation}
Using Lemma \ref{iotan lemma}, we have 
\begin{multline*}
    \|\iota_n(T(A))\otimes |x|\|_{E(\mathcal{M}_n,\tau_n)}=\||\iota_n(T(A))|\otimes |x|\|_{E(\mathcal{M}_n,\tau_n)}\\  =\|U|T(A)|U^{\ast}\otimes E_{1,1}\otimes |x|\|_{E(\mathcal{M}_n,\tau_n)}=\|T(A)\otimes x\|_{E(B(\ell_2)\overline{\otimes}\mathcal{M},{\rm Tr}\otimes\tau)} 
\end{multline*}
and, similarly, \begin{multline*}
    \|\iota_n(A-{\rm diag}(A))\otimes |x|\|_{E(\mathcal{M}_n,\tau_n)}=\||\iota_n(A-{\rm diag}(A))|\otimes |x|\|_{E(\mathcal{M}_n,\tau_n)}\\=\|U|A-{\rm diag}(A)|U^{\ast}\otimes E_{1,1}\otimes |x|\|_{E(\mathcal{M}_n,\tau_n)}=\|(A-{\rm diag}(A))\otimes x\|_{E(B(\ell_2)\overline{\otimes}\mathcal{M},{\rm Tr}\otimes\tau)}\\{\leq}2\|A\otimes x\|_{E(B(\ell_2)\overline{\otimes}\mathcal{M},{\rm Tr}\otimes\tau)},
\end{multline*}
 where the latter inequality follows from Lemma \ref{peter lemma}. Combining these two estimates with \eqref{eq_estimate}, we arrive at the desired conclusion and thus complete the proof. 
\end{proof}

\begin{lem}\label{beta fedor lemma} Let $E$ be a symmetric function or sequence space. If $\beta_E=1,$ then, for every $n\in\mathbb{N},$ there exists $x_n\in E$ so that
$$\|\alpha\|_1\|x_n\|_E\leq 2\|\alpha\otimes x_n\|_{ E},\quad \alpha\in \ell_{\infty}^n.$$
\end{lem}
\begin{proof}
{ Suppose that $\beta_E = 1$. The mapping $n \mapsto \|D_n\|_{E \to E}$, $n \in \mathbb{N}$, is known to be submultiplicative (\cite[Theorem 4.5]{KPS}) and dominated by the identity mapping, that is $\|D_n\|_{E \to E}\le n$, $n \in \mathbb{N}$  (\cite[Corollary 1, p.~98]{KPS}). 

Assume, towards a contradiction, that $\|D_n\|_{E \to E} < n$ for some $n \in \mathbb{N}$. Then there exists $\gamma < 1$ such that $\|D_n\|_{E \to E} \le n^\gamma.$

Using submultiplicativity,   for every $k \in \mathbb{N}$ we obtain
\[
\|D_{n^k}\|_{E \to E} \le \|D_n\|_{E \to E}^k \le (n^\gamma)^k = (n^k)^\gamma.
\]
This shows that
\[
\frac{\log \|D_{n^k}\|_{E \to E}}{\log n^k} \le \gamma \quad \text{for all } k.
\]
Taking the limit as $k \to \infty$, we obtain:
\[
\limsup_{m \to \infty} \frac{\log \|D_m\|_{E \to E}}{\log m} \le \gamma < 1,
\]
which contradicts the assumption $\beta_E = 1$. Hence, we conclude that
\[
\|D_n\|_{E \to E} = n \quad \text{for all } n \in \mathbb{N}.
\]}
Next for $n\in\mathbb{N},$ { select} $x_n\in E$ such that $\|D_nx_n\|_E\geq\frac{n}{2}\|x_n\|_E.$ 

Let $\alpha\in \ell_{\infty}^n.$ Set $\alpha_k\in \ell_{\infty}^n,$ $1\leq k\le n,$ be sequences supported on the set $\{1,\cdots,n\}$ which are cyclic permutations of $\mu(\alpha)$, that is $$\alpha_k(j)=\mu((j-k){\rm mod} \, n+1,\alpha), \quad 1\leq j\le n .$$ By triangle inequality, we have
$$n\|\alpha\otimes x_n\|_{E}=\sum_{k=1}^{n}\|\alpha_k\otimes x_n\|_{E}\geq\Big\|\sum_{k=1}^{n}\alpha_k\otimes x_n\Big\|_{E}.$$
However,
$$\sum_{k=1}^{n}\alpha_k=\big(\underbrace{\|\alpha\|_1, \dots, \|\alpha\|_1}_{n \text{ times}}\big)=\|\alpha\|_1\big(\underbrace{1, \dots, 1}_{n \text{ times}}\big).$$
Hence,
\begin{equation*}
    n\|\alpha\otimes x_n\|_{E}\geq \|\alpha\|_1\|\big(\underbrace{1, \dots, 1}_{n \text{ times}}\big)\otimes x_n\|_{E}=\|\alpha\|_1\|D_nx_n\|_E\geq\frac{n}{2}\|\alpha\|_1\|x_n\|_E.
\end{equation*}
Dividing by $n,$ we complete the proof.
\end{proof}

\begin{lem}\label{alpha fedor lemma} Let $E$ be a symmetric function or sequence space. If $\alpha_E=0,$ then, for every $n\in\mathbb{N},$ there exists $x_n\in E$ so that
$$\|\alpha\otimes x_n\|_E\leq 2\|\alpha\|_{\infty}\|x_n\|_E,\quad \alpha\in \ell_{\infty}^n.$$
\end{lem}
\begin{proof} {   Suppose that $\alpha_E=0$. The mapping $n \mapsto \|D_{1/n}\|_{E \to E}$, $n \in \mathbb{N}$, is known to be submultiplicative (\cite[Theorem 4.5]{KPS}). Setting $\omega_n=\|D_{1/n}\|_{E \to E}$, we have that if $\omega_n<1$ for some $n,$ it follows that $\omega_n=O(n^{-\gamma})$ for some $\gamma>0.$}

{ 
Indeed, let $\omega_{n_0}<1$ for some $n_0\ge 2$. Taking an arbitrary $n\ge 1$, we denote $k=\lfloor\frac{\ln n}{\ln n_0}\rfloor$ so that $n_0^k\le n<n_0^{k+1}$. For a function space, take $r=\frac{n}{n_0^k}$ with $1\le r < n_0$, we obtain $$\omega_n=\omega_{n_0^k r}\le \omega_{n_0^k}\omega_r\le C_0 \omega_{n_0^k},$$ where $C_0=\max_{1\le r<n_0}\omega_r$. On the other hand, $$\omega_{n_0^k}\le \omega_{n_0}^k\le (n_0^k)^{-\gamma}=\Big(\frac{n}{r}\Big)^{-\gamma}=r^\gamma n^{-\gamma}< n_0^\gamma n^{-\gamma},$$ where $\gamma =-\ln \omega_{n_0}/\ln n_0>0.$ Taking now $C=C_0 n_0^\gamma$ we conclude that $\omega_n\le C n^{-\gamma},$ for $\gamma>0$ or $\omega_n=O(n^{-\gamma}).$ For a sequence space, the argument is similar, due to the natural identification of a sequence space with the subspace of characteristic functions in a function space.

Thus, we have that
$$
\alpha_E = \lim_{n\to\infty} \frac{\log\omega_n }{\log(1/n) }=\lim_{n\to\infty}\frac{-\log\omega_n }{\log n }\ge \lim_{n\to\infty} \frac{-\log C n^{-\gamma} }{\log n }=\lim_{n\to\infty} \Big(\gamma-\frac{\log C}{\log n }\Big)>0,
$$
which contradicts the assumption $\alpha_E=0.$ Hence, $\omega_n=\|D_{1/n}\|_{E \to E}\ge 1$ for all $n\in\mathbb{N}.$ 

Since $$\|D_{1/n}\|_{E \to E}=\sup_{\|y\|_E= 1}\|D_{1/n}y\|_E\ge 1>\frac12,$$ it follows that we can select $y_n\in E$ with $\|y_n\|_E=1$ and  $\|D_{1/n}y_n\|_E\ge \frac12.$ Let $x_n=D_{1/n}y_n$.  In particular, we have 
$$
D_nx_n=D_n D_{1/n}y_n=y_n
$$ 
and so $\|D_nx_n\|_E=\|y_n\|_E=1.$ On the other hand, 
$$
\|x_n\|_E=\|D_{1/n}y_n\|_E\ge  \frac12=\frac12\|D_nx_n\|_E.
$$    
Hence, for $n\in\mathbb{N},$ we can select $x_n\in E$ such that $\|D_nx_n\|_E\leq 2\|x_n\|_E.$ 

Now consider $\alpha\in \ell_{\infty}^n$.
Since
$$\mu(\alpha)\leq\|\alpha\|_{\infty}\big(\underbrace{1, \dots, 1}_{n \text{ times}}\big),$$
it follows that
$$\mu(\alpha\otimes x_n)\leq \|\alpha\|_{\infty} \mu\Big(\big(\underbrace{1, \dots, 1}_{n \text{ times}}\big) \otimes x_n\Big)=\|\alpha\|_{\infty} D_n\mu(x_n).$$
Taking $\|\cdot\|_E$ of the both sides of the inequality above, it follows that
$$\|\alpha\otimes x_n\|_E\leq \|\alpha\|_{\infty}\|D_nx_n\|_E\le 2\|\alpha\|_{\infty}\|x_n\|_.$$}
\end{proof}

\begin{lem}\label{lem_upper} Let $(\mathcal{M},\tau)$ be a semifinite infinite factor. Let $E$ be a symmetric function or sequence space (depending on whether $\mathcal{M}$ is II$_{\infty}$ or I$_{\infty}$). If 
$$\|T(A)\otimes x\|_{E(B(\ell_2)\overline{\otimes}\mathcal{M},{\rm Tr}\otimes\tau)}\leq 2c_{E,\mathcal{M},\tau}\|A\otimes x\|_{E(B(\ell_2)\overline{\otimes}\mathcal{M},{\rm Tr}\otimes\tau)}$$
for every $x\in E(\mathcal{M},\tau)$ and for every matrix $A$ having only finitely many non-zero entries, then $\beta_E<1.$
\end{lem}
\begin{proof} Suppose $\beta_E=1.$ By Lemma \ref{beta fedor lemma}, for every $n\in\mathbb{N},$ we can find non-zero $x_n\in E(\mathcal{M},\tau)$ such that
$$\frac12\|A\|_{L_1(B(\ell_2),{\rm Tr})}\|x_n\|_{E(\mathcal{M},\tau)}\leq\|A\otimes x_n\|_{E(B(\ell_2)\overline{\otimes}\mathcal{M},{\rm Tr}\otimes\tau)},\quad {\rm rank}(A)\leq n.$$

Let now $A=T(P),$ where $P\in M_n(\mathbb{C})$ is a rank one projection. We have
$$\frac12\|T(P)\|_{L_1(M_n(\mathbb{C}),{\rm Tr})}\|x_n\|_{E(\mathcal{M},\tau)}\leq \|T(P)\otimes x_n\|_{E(B(\ell_2)\overline{\otimes}\mathcal{M},{\rm Tr}\otimes\tau)}$$
$$\leq 2c_{E,\mathcal{M},\tau}\|P\otimes x_n\|_{E(B(l_2)\overline{\otimes}\mathcal{M},{\rm Tr}\otimes\tau)}=2c_{E,\mathcal{M},\tau}\|x_n\|_{E(\mathcal{M},\tau)}.$$
it follows that
$$\|T(P)\|_{L_1(M_n(\mathbb{C}),{\rm Tr})}\leq 4c_{E,\mathcal{M},\tau}$$
for every rank one projection $P\in M_n(\mathbb{C}).$ However, { this contradicts Davies \cite[Lemma 10 and Corollary 11]{Davies1988}, where he shows that triangular truncation $T$ is unbounded on $\mathcal{S}_1$ {(see also \cite[Chapter 4]{Davidson-Nest})}. In particular, he considered the projection 
\[
P_n = \tfrac{1}{n} J_n, \qquad J_n = (1)_{i,j=1}^n,
\] 
onto the span of $(1,\dots,1)$, and showed that 
\[
\|T(P_n)\|_{1} \sim \log n.
\] Thus no bound of the form $\|T(P)\|_{1} \le C$ (independent of $n$) can be valid even on rank-one projections.
 Hence, our initial assumption is false.} We conclude that $\beta_E<1.$
\end{proof}

\begin{lem}\label{lem_lower} Let $(\mathcal{M},\tau)$ be a semifinite infinite factor. Let $E$ be a symmetric function or sequence space (depending on whether $\mathcal{M}$ is II$_{\infty}$ or I$_{\infty}$). If 
$$\|T(A)\otimes x\|_{E(B(\ell_2)\overline{\otimes}\mathcal{M},{\rm Tr}\otimes\tau)}\leq 2c_{E,\mathcal{M},\tau}\|A\otimes x\|_{E(B(\ell_2)\overline{\otimes}\mathcal{M},{\rm Tr}\otimes\tau)}$$
for every $x\in E(\mathcal{M},\tau)$ and for every matrix $A$ having only finitely many non-zero entries, then $\alpha_E>0.$
\end{lem}
\begin{proof} Suppose $\alpha_E=0.$ By Lemma \ref{alpha fedor lemma}, for every $n\in\mathbb{N},$ we can find non-zero $x_n\in E(\mathcal{M},\tau)$ such that
$$\|A\otimes x_n\|_{E(B(\ell_2)\overline{\otimes}\mathcal{M},{\rm Tr}\otimes\tau)}\leq 2\|A\|_{B(\ell_2)}\|x\|_{E(\mathcal{M},\tau)},\quad {\rm rank}(A)\leq n.$$

Next we note that for any $B\in M_n(\mathbb{C})$ and $x\in E(\mathcal{M},\tau)$, we have
$$\|B\otimes x\|_{E(B(\ell_2)\overline{\otimes}\mathcal{M},{\rm Tr}\otimes\tau)}=\|\mu(B)\otimes x\|_{E(\ell_{\infty}\overline{\otimes}\mathcal{M},\Sigma\otimes\tau)}\geq$$
$$\geq\|\mu(1,B)e_1\otimes x\|_{E(\ell_{\infty}\overline{\otimes}\mathcal{M},\Sigma\otimes\tau)}=\mu(1,B)\|x\|_{E(\mathcal{M},\tau)}=\|B\|_{B(\ell_2)}\|x\|_{E(\mathcal{M},\tau)},$$ where $\Sigma$ is a counting measure on $\mathbb{N}$, $\mu(1,B)=\|B\|_{B(\ell_2)}$ is the maximal singular value of the matrix $B$ and $e_1=(1,0,\ldots,0,\ldots)$ is the first standard vector basis.

Take now $A\in M_n(\mathbb{C}).$ Applying the latter inequality with $B=T(A)$ and $x=x_n$ yields
$$\|T(A)\|_{B(\ell_2)}\|x_n\|_{E(\mathcal{M},\tau)}\leq \|T(A)\otimes x_n\|_{E(B(\ell_2)\overline{\otimes}\mathcal{M},{\rm Tr}\otimes\tau)}\leq$$
$$\leq 2c_{E,\mathcal{M},\tau}\|A\otimes x_n\|_{E(B(\ell_2)\overline{\otimes}\mathcal{M},{\rm Tr}\otimes\tau)}\le 4c_{E,\mathcal{M},\tau}\|A\|_{B(\ell_2)}\|x_n\|_{E(\mathcal{M},\tau)}.$$
It follows that
$$\|T(A)\|_{B(\ell_2)}\leq 4c_{E,\mathcal{M},\tau}\|A\|_{B(\ell_2)},\quad A\in M_n(\mathbb{C}).$$
 This contradicts the result of Davies \cite[Lemma~10 and Corollary~11]{Davies1988}, 
where he constructed finite-dimensional matrices $A_n\in M_n(\mathbb C)$ such that 
$\|T(A_n)\|_{1}\sim \log n$. By duality argument no uniform bound of the form 
$\|T(A)\|_{B(\ell_2)}\le C\|A\|_{B(\ell_2)}$ with $C$ independent of $n$ can hold {(see also \cite[Chapter 4]{Davidson-Nest})}. Hence, our initial assumption is false. We conclude that $\alpha_E>0.$
\end{proof}

\section{Proof of necessity for type II$_1$ factors} \label{sec_II_1}

In this section we prove that the conditions stated in Theorem \ref{MainThm_OpSp} are necessary in the factor of type II$_{1}$ assuming that $\tau({\rm id}_{\mathcal{M}})=1$.

For $m\in\mathbb{N}$ let $E_m$ be a symmetric function space on $(0,m)$ defined by the setting
$$\|f\|_{E_m}=\|D_{\frac1m}f\|_E.$$

\begin{lem}\label{lem_another side_fin} Let $\mathcal{M}$ be a finite factor equipped with a faithful normal tracial state $\tau.$ Let $E$ be a symmetric function space on $(0,1).$ If
$$\|[X^{\ast},Y]\|_{E(\mathcal{M},\tau)}\leq c_{E,\mathcal{M},\tau}\|[X,Y]\|_{E(\mathcal{M},\tau)}$$
for every normal $X\in E(\mathcal{M},\tau)$ and for every $Y\in \mathcal{M},$ then
$$\|T(A)\otimes x\|_{E_{(n+1)^2}(M_{(n+1)^2}(\mathbb{C})\otimes\mathcal{M},{\rm Tr}\otimes\tau)}\leq$$
$$\leq 2c_{E,\mathcal{M},\tau}\|A\otimes x\|_{E_{(n+1)^2}(M_{(n+1)^2}(\mathbb{C})\otimes\mathcal{M},{\rm Tr}\otimes\tau)}.$$
for every $x\in E(\mathcal{M},\tau)$ and for every $A\in M_n(\mathbb{C}).$ \end{lem}
\begin{proof} In what follows, we need a probability space
$$(\mathcal{M}_n',\tau_n')\stackrel{def}{=}(M_{n+1}(\mathbb{C})\otimes M_{n+1}(\mathbb{C})\otimes\mathcal{M},(n+1)^{-2}{\rm Tr}\otimes{\rm Tr}\otimes\tau)$$
and a non-commutative measure space
$$(\mathcal{M}_n,\tau_n)\stackrel{def}{=}(M_{n+1}(\mathbb{C})\otimes M_{n+1}(\mathbb{C})\otimes\mathcal{M},{\rm Tr}\otimes{\rm Tr}\otimes\tau).$$
Since the non-commutative measure spaces $(\mathcal{M}_n',\tau_n')$ and $(\mathcal{M},\tau)$ are {  $*$-isomorphic via the
identity on the underlying $*$-algebra ($\mathcal M'_n=\mathcal M_n$),
while the traces differ by a constant factor $\tau'_n=(n+1)^{-2}\,\tau_n.$}, it follows from the assumption that
$$\|[X^{\ast},Y]\|_{E_{(n+1)^2}(\mathcal{M}_n,\tau_n)}=\|[X^{\ast},Y]\|_{E(\mathcal{M}_n',\tau_n')}\leq$$
$$\leq c_{E,\mathcal{M},\tau}\|[X,Y]\|_{E(\mathcal{M}_n',\tau_n')}=c_{E,\mathcal{M},\tau}\|[X,Y]\|_{E_{(n+1)^2}(\mathcal{M}_n,\tau_n)}$$
for every normal $X\in E_{(n+1)^2}(\mathcal{M}_n,\tau_n)$ and for every $Y\in\mathcal{M}_n.$

Substituting $X=A_m^n\otimes |x|$ and $Y=B_m^n\otimes 1,$ where $A_m^n$ and $B_m^n$ are constructed in Lemma \ref{commutator limit lemma}, we obtain
$$\|[(A_m^n)^{\ast},B_m^n]\otimes |x|\|_{E_{(n+1)^2}(\mathcal{M}_n,\tau_n)}\leq c_{E,\mathcal{M},\tau}\|[A_m^n,B_m^n]\otimes |x|\|_{E_{(n+1)^2}(\mathcal{M}_n,\tau_n)}.$$
Using Lemma \ref{commutator limit lemma} and passing $m\to\infty,$ we infer
$$\|\iota_n(T(A))\otimes |x|\|_{E_{(n+1)^2}(\mathcal{M}_n,\tau_n)}\leq c_{E,\mathcal{M},\tau}\|\iota_n(A-{\rm diag}(A))\otimes |x|\|_{E_{(n+1)^2}(\mathcal{M}_n,\tau_n)}.$$
Using Lemma \ref{iotan lemma}, we write
\begin{align*}
    \|\iota_n(T(A))\otimes |x|\|_{E_{(n+1)^2}(\mathcal{M}_n,\tau_n)}&=\||\iota_n(T(A))|\otimes |x|\|_{E_{(n+1)^2}(\mathcal{M}_n,\tau_n)}\\ &
    =\|U|T(A)|U^{\ast}\otimes E_{1,1}\otimes |x|\|_{E_{(n+1)^2}(\mathcal{M}_n,\tau_n)}\\ &
    =\|T(A)\otimes x\|_{E_{(n+1)^2}(M_{(n+1)^2}(\mathbb{C})\overline{\otimes}\mathcal{M},{\rm Tr}\otimes\tau)}.
\end{align*} On the other hand, we have
\begin{align*}
    \|\iota_n(A-{\rm diag}(A)) &\otimes |x|\|_{E_{(n+1)^2}(\mathcal{M}_n,\tau_n)}\\&=\||\iota_n(A-{\rm diag}(A))|\otimes |x|\|_{E_{(n+1)^2}(\mathcal{M}_n,\tau_n)}\\ & =\|U|A-{\rm diag}(A)|U^{\ast}\otimes E_{1,1}\otimes |x|\|_{E_{(n+1)^2}(\mathcal{M}_n,\tau_n)}\\ &=\|(A-{\rm diag}(A))\otimes x\|_{E_{(n+1)^2}(M_{(n+1)^2}(\mathbb{C})\overline{\otimes}\mathcal{M},{\rm Tr}\otimes\tau)}\\ & \stackrel{L.\ref{peter lemma}}{\leq}2\|A\otimes x\|_{E_{(n+1)^2}(M_{(n+1)^2}(\mathbb{C})\overline{\otimes}\mathcal{M},{\rm Tr}\otimes\tau)}.
\end{align*}
Combining the last three displays, we complete the proof.
\end{proof}

{  For a positive matrix $A\in M_n(\mathbb C)$,
by ${\rm supp}(A)$ we denote
the orthogonal projection onto ${\rm Ran}A$.
For a general $B\in M_n(\mathbb C)$, 
by ${\rm supp}(B)$ we understand ${\rm supp}(|B|)$. Note that ${\rm Tr}({\rm supp}(B))={\rm rank}(B).$

}
\begin{lem}\label{tensor fact} Let $\mathcal{M}$ be a finite factor equipped with a faithful normal tracial state $\tau.$ Let $F$ be a symmetric function space on $(0,m).$ For every $B\in M_m(\mathbb{C})$ and for every $x\in F(\mathcal{M},\tau)$ we have

\itemeq{i}{\|B\|_{M_m(\mathbb{C})}\|x\|_{F(\mathcal{M},\tau)}\leq\|B\otimes x\|_{F(M_m(\mathbb{C})\otimes\mathcal{M},{\rm Tr}\otimes\tau)},}
\itemeq{ii}
{{\|B\|_{L_1(M_m(\mathbb{C}),{\rm Tr})}}\|{\rm supp}(B)\otimes x\|_{F(M_m(\mathbb{C})\otimes\mathcal{M},{\rm Tr}\otimes\tau)}} 
\itemeqnext[10em]{\leq \|B\otimes x\|_{F(M_m(\mathbb{C})\otimes\mathcal{M},{\rm Tr}\otimes\tau)}{{\rm Tr}({\rm supp}(B))},}
\itemeq{iii}{\|B\otimes x\|_{F(M_m(\mathbb{C})\otimes\mathcal{M},{\rm Tr}\otimes\tau)}\leq\|B\|_{L_1(M_m(\mathbb{C}),{\rm Tr})}\|x\|_{F(\mathcal{M},\tau)},}
\itemeq{iv}{\|B\otimes x\|_{F(M_m(\mathbb{C})\otimes\mathcal{M},{\rm Tr}\otimes\tau)}\leq\|B\|_{M_m(\mathbb{C})}\|1\otimes x\|_{F(\mathbb{C}^m\otimes\mathcal{M},\Sigma\otimes\tau)}.}
\end{lem}
\begin{proof} 
Let $B\in M_m(\mathbb C)$. Recall that
\[
\mu(B)=\big(\mu(1,B),\mu(2,B),\dots,\mu(m,B)\big)
\]
is a sequence the singular numbers of $B$ in non-increasing order.
Fix $x\in F(\mathcal M,\tau)$ and recall that $\Sigma$ is the counting measure on $\{1,\dots,m\}$.

(i). 
Since the projection on the first coordinate is a positive contraction and $\mu(1,B)$ is the largest singular number, that is $\mu(1,B)=\|B\|_{M_m(\mathbb C)}$, we have
\begin{multline*}
    \|B\otimes x\|_{F(M_m(\mathbb{C})\otimes\mathcal{M},{\rm Tr}\otimes\tau)}=\|\mu(B)\otimes x\|_{F(\mathbb{C}^m\otimes\mathcal{M},\Sigma\otimes\tau)}\\ \geq \|\mu(1,B)e_1\otimes x\|_{F(\mathbb{C}^m\otimes\mathcal{M},\Sigma\otimes\tau)}=\|B\|_{M_m(\mathbb{C})}\|x\|_{F(\mathcal{M},\tau)},
\end{multline*}
which proves the first assertion.

    (ii).
Denote $n={\rm Tr}({\rm supp}(B)).$ Then $\mu(B)=\big(\mu(1,B),\dots,\mu(n,B),0,\ldots,0\big)$ and $n\le m$. Let $\alpha_k,$ $1\leq k\le n,$ be sequences supported on the set $\{1,\ldots,n\}$ which are cyclic permutations of the first $n$ coordinates of $\mu(B),$ that is $$\alpha_k(j)=\left\{\begin{array}{rr}\mu(((j-k){\bmod} \, n)+1,B), & \quad 1\leq j\le n \\ 0 , & \quad n+1\leq j\le m \end{array}\right..$$
We have that $$\sum_{k=1}^n \alpha_k= (\underbrace{1, \dots, 1}_{n \text{ times}}, 0, \ldots, 0)\sum_{i=1}^n \mu(i,B)=\|B\|_{L_1(M_m(\mathbb{C}),{\rm Tr})}{\rm supp}(B).$$
Hence, 
\begin{multline*}
    n\|B\otimes x\|_{F(M_m(\mathbb{C})\otimes\mathcal{M},{\rm Tr}\otimes\tau)}=n\|\mu(B)\otimes x\|_{F(\mathbb{C}^m\otimes\mathcal{M},\Sigma\otimes\tau)}\\= \sum_{k=1}^{n}\|\alpha_k\otimes x\|_{F(\mathbb{C}^m\otimes\mathcal{M},\Sigma\otimes\tau)}\geq \big\|\sum_{k=1}^{n}\alpha_k\otimes x\big\|_{F(\mathbb{C}^m\otimes\mathcal{M},\Sigma\otimes\tau)}\\=\|B\|_{L_1(M_m(\mathbb{C}),{\rm Tr})}\big\|{\rm supp}(B)\otimes x\big\|_{F(\mathbb{C}^m\otimes\mathcal{M},\Sigma\otimes\tau)}.
\end{multline*}
So we infer the second assertion.

(iii). Similarly, 
\begin{align*}
    \|B\otimes x\|_{F(M_m(\mathbb{C})\otimes\mathcal{M},{\rm Tr}\otimes\tau)}&=\|\mu(B)\otimes x\|_{F(\mathbb{C}^m\otimes\mathcal{M},\Sigma\otimes\tau)}\\&=\Big\|\sum_{k=1}^{m}\mu(k,B)e_k\otimes x\Big\|_{F(\mathbb{C}^m\otimes\mathcal{M},\Sigma\otimes\tau)}\\&\leq \sum_{k=1}^{m}\mu(k,B)\|e_k\otimes x\|_{F(\mathbb{C}^m\otimes\mathcal{M},\Sigma\otimes\tau)}\\&=\|B\|_{L_1(M_m(\mathbb{C}),{\rm Tr})}\|x\|_{F(\mathcal{M},\tau)},
\end{align*}
which proves the third assertion.

(iv). Finally,
\begin{multline*}
    \|B\otimes x\|_{F(M_m(\mathbb{C})\otimes\mathcal{M},{\rm Tr}\otimes\tau)}=\|\mu(B)\otimes x\|_{F(\mathbb{C}^m\otimes\mathcal{M},\Sigma\otimes\tau)}\\ \leq  \mu(1,B)\| 1\otimes x\|_{F(\mathbb{C}^m\otimes\mathcal{M},\Sigma\otimes\tau)}=\|B\|_{M_m(\mathbb{C})}\|1\otimes x\|_{F(\mathbb{C}^m\otimes\mathcal{M},\Sigma\otimes\tau)}.
\end{multline*}
This proves the fourth assertion.
\end{proof}

{  For brevity, below we use $\|\cdot\|_{E\circlearrowleft}$ instead of $\|\cdot\|_{E\to E}$.}

\begin{lem}\label{finite beta lemma} Let $\mathcal{M}$ be a finite factor equipped with a faithful normal tracial state $\tau$. Let $E$ be a symmetric function space on $(0,1).$ If
\begin{multline*}
    \|T(A)\otimes x\|_{E_{(n+1)^2}(M_{(n+1)^2}(\mathbb{C})\otimes\mathcal{M},{\rm Tr}\otimes\tau)}\\\leq 2c_{E,\mathcal{M},\tau}\|A\otimes x\|_{E_{(n+1)^2}(M_{(n+1)^2}(\mathbb{C})\otimes\mathcal{M},{\rm Tr}\otimes\tau)}.
\end{multline*}
for every $x\in E(\mathcal{M},\tau)$ and $A\in M_n(\mathbb{C}),$ then
$$\|T\|_{L_1(M_n(\mathbb{C}),{\rm Tr})\circlearrowleft}\|D_{(n+1)^2}\|_{E\circlearrowleft}\leq  4n(n+1)c_{E,\mathcal{M},\tau}.$$
\end{lem}
\begin{proof}

Fix $f\in E(0,1)$ supported on $(0,(n+1)^{-2})$. {In this case $D_{(n+1)^2}f\in E(0,1)$.} So choose $x\in E(\mathcal{M},\tau)$ such that $D_{(n+1)^2}\mu(f)=\mu(x).$ Let now $A\in M_n(\mathbb{C})$ and $r(A)={\rm Tr}({\rm supp}(T(A)))\le n$. We will assume that $r(A)\ge 1$.  We have that
\begin{align*}
    \|T(A)& \|_{L_1(M_n(\mathbb{C}),{\rm Tr})}\|D_{r(A)}f\|_E\\&=\|T(A)\|_{L_1(M_n(\mathbb{C}),{\rm Tr})}\|{\rm supp}(T(A))\otimes x\|_{E_{(n+1)^2}(M_{(n+1)^2}(\mathbb{C})\otimes\mathcal{M},{\rm Tr}\otimes\tau)}\\ & \stackrel{L.\ref{tensor fact}, {\rm (ii)}}{\leq}r(A)\|T(A)\otimes x\|_{E_{(n+1)^2}(M_{(n+1)^2}(\mathbb{C})\otimes\mathcal{M},{\rm Tr}\otimes\tau)}\\ & \leq 2r(A)c_{E,\mathcal{M},\tau}\|A\otimes x\|_{E_{(n+1)^2}(M_{(n+1)^2}(\mathbb{C})\otimes\mathcal{M},{\rm Tr}\otimes\tau)}\\ &\stackrel{L.\ref{tensor fact}, {\rm (iii)}}{\leq}2r(A)c_{E,\mathcal{M},\tau}\|A\|_{L_1(M_n(\mathbb{C}),{\rm Tr})}\| x\|_{E_{(n+1)^2}(\mathcal{M},\tau)}\\ &=2r(A)c_{E,\mathcal{M},\tau}\|A\|_{L_1(M_n(\mathbb{C}),{\rm Tr})}\|f\|_E.
\end{align*}
 Taking into account that
\begin{multline*}
\|D_{(n+1)^2}f\|_E\leq(n+1)\|D_{n+1}f\|_E\leq(n+1)\|D_{2n}f\|_E\leq 2(n+1)\|D_nf\|_E\\=2(n+1)\|D_{\frac{n}{r(A)}} D_{r(A)}f\|_E\le \frac{2n(n+1)}{r(A)} \|D_{r(A)}f\|_E,
\end{multline*} we obtain
\begin{align*}
    \|T(A)\|_{L_1(M_n(\mathbb{C}),{\rm Tr})}\|D_{(n+1)^2}f\|_E & \le \frac{2n(n+1)}{r(A)}\|T(A)\|_{L_1(M_n(\mathbb{C}),{\rm Tr})}\|D_{r(A)}f\|_E \\ & \le \frac{2n(n+1)}{r(A)}2r(A)c_{E,\mathcal{M},\tau}\|A\|_{L_1(M_n(\mathbb{C}),{\rm Tr})}\|f\|_E\\ &=4n(n+1)c_{E,\mathcal{M},\tau}\|A\|_{L_1(M_n(\mathbb{C}),{\rm Tr})}\|f\|_E.
\end{align*} Taking the supremum over $A\in L_1(M_n(\mathbb{C}),{\rm Tr}),$ we write
$$\|T\|_{L_1(M_n(\mathbb{C}),{\rm Tr})\circlearrowleft}\|D_{(n+1)^2}f\|_E\leq 4n(n+1)c_{E,\mathcal{M},\tau}\|f\|_E$$
for every $f\in E(0,1)$ supported on $(0,(n+1)^{-2}).$
 This completes the proof.
\end{proof}

\begin{lem}\label{finite alpha lemma} Let $\mathcal{M}$ be a finite factor equipped with a faithful normal tracial state $\tau.$ Let $E$ be a symmetric function space on $(0,1).$ If
\begin{multline*}
    \|T(A)\otimes x\|_{E_{(n+1)^2}(M_{(n+1)^2}(\mathbb{C})\otimes\mathcal{M},{\rm Tr}\otimes\tau)} \\ \leq 2c_{E,\mathcal{M},\tau}\|A\otimes x\|_{E_{(n+1)^2}(M_{(n+1)^2}(\mathbb{C})\otimes\mathcal{M},{\rm Tr}\otimes\tau)}.
\end{multline*}
for every $x\in E(\mathcal{M},\tau)$ and for every $A\in M_n(\mathbb{C}),$ then
$$\|T\|_{M_n(\mathbb{C})\circlearrowleft}\|D_{(n+1)^{-2}}\|_{E\circlearrowleft}\leq 2c_{E,\mathcal{M},\tau}.$$
\end{lem}
\begin{proof} Fix $f\in E(0,1)$ and choose $x\in E(\mathcal{M},\tau)$ such that $\mu(f)=\mu(x).$ We have
\begin{equation}\label{eq_lem24}\begin{aligned}
\|T(A)\|_{M_n(\mathbb{C})}&\|D_{(n+1)^{-2}}f\|_E=\|T(A)\|_{M_n(\mathbb{C})}\|x\|_{E_{(n+1)^2}(\mathcal{M},\tau)}\\ &\stackrel{L.\ref{tensor fact}, {\rm (i)}}{\leq}\|T(A)\otimes x\|_{E_{(n+1)^2}(M_{(n+1)^2}(\mathbb{C})\otimes\mathcal{M},{\rm Tr}\otimes\tau)}\\ & \leq 2c_{E,\mathcal{M},\tau}\|A\otimes x\|_{E_{(n+1)^2}(M_{(n+1)^2}(\mathbb{C})\otimes\mathcal{M},{\rm Tr}\otimes\tau)}\\ &\stackrel{L.\ref{tensor fact}, {\rm (iv)}}{\leq}2c_{E,\mathcal{M},\tau}\|A\|_{M_n(\mathbb{C})}\|1\otimes x\|_{E_{(n+1)^2}(\mathbb{C}^{(n+1)^2}\otimes\mathcal{M},\Sigma\otimes\tau)}\\ &=2c_{E,\mathcal{M},\tau}\|A\|_{M_n(\mathbb{C})}\|f\|_E, 
\end{aligned}\end{equation} where the last equality appears due to $$
\|1\otimes x\|_{E_{(n+1)^2}(\mathbb C^{(n+1)^2}\otimes\mathcal M,\Sigma\otimes\tau)}
=  \big\|D_{(n+1)^{-2}}D_{(n+1)^2}\mu(x)\big\|_{E}
= \|x\|_{E}.
$$
Taking the supremum in \eqref{eq_lem24} over $A\in M_n(\mathbb{C})$ and $f\in E(0,1),$ we complete the proof.
\end{proof}

\begin{lem}\label{lem_indices_fin} Let $\mathcal{M}$ be a finite factor equipped with a faithful normal tracial state $\tau.$ Let $E$ be a symmetric function space on $(0,1).$ Assume that for every $n\in\mathbb N$,
\begin{multline}
    \label{eq:prod-bounds}\|T\|_{M_n(\mathbb{C})\circlearrowleft}\|D_{(n+1)^{-2}}\|_{E\circlearrowleft}\leq 2c_{E,\mathcal{M},\tau}, \\  
\|T\|_{L_1(M_n(\mathbb{C}),{\rm Tr})\circlearrowleft}\|D_{(n+1)^2}\|_{E\circlearrowleft}\leq  4n(n+1)c_{E,\mathcal{M},\tau}, \end{multline}
then $\alpha_E>0$ and $\beta_E<1.$
\end{lem}
\begin{proof} By Davies \cite[Lemma~10 and Corollary~11]{Davies1988}, we have that $\|T\|_{L_1(M_n(\mathbb C),{\rm Tr})\circlearrowleft}\to\infty$ as $n\to\infty$. The duality argument shows that  $\|T\|_{M_n(\mathbb C)\circlearrowleft}\to\infty$ as $n\to\infty$ {(see  \cite[Chapter 4]{Davidson-Nest})}.

From the given product bounds \eqref{eq:prod-bounds} we deduce
\[
\|D_{(n+1)^{-2}}\|_{E\circlearrowleft}\ =\ o(1),
\qquad
\|D_{(n+1)^2}\|_{E\circlearrowleft}\ =\ o\big((n+1)^2\big),
\qquad n\to\infty.
\]
In particular, there exists $n_0\in\mathbb N$ such that
\[
\|D_{(n_0+1)^{-2}}\|_{E\circlearrowleft}<1,
\qquad
\|D_{(n_0+1)^2}\|_{E\circlearrowleft}<(n_0+1)^2.
\]
Set $m_0:=(n_0+1)^2>1$. Then
\[
\|D_{m_0^{-1}}\|_{E\circlearrowleft}<1,
\qquad
\|D_{m_0}\|_{E\circlearrowleft}<m_0.
\]
Consequently, we can choose numbers $\alpha,\beta\in(0,1)$ such that
\begin{equation*}\label{eq:node-bounds}
\|D_{m_0^{-1}}\|_{E\circlearrowleft}\ \le\ m_0^{-\alpha},
\qquad
\|D_{m_0}\|_{E\circlearrowleft}\ \le\ m_0^{\beta}.
\end{equation*}
It follows that for every integer $k$,
\[
\|D_{m_0^k}\|_{E\circlearrowleft}\ \le\
\begin{cases}
m_0^{k\beta}, & k\ge 0,\\[1mm]
m_0^{k\alpha}, & k\le 0.
\end{cases}
\]
Now fix $s>0$ and choose $k\in\mathbb Z$ with $s\in[m_0^k,m_0^{k+1}]$. If $s\ge 1$ (hence $k\ge 0$), then by monotonicity,
\[
\|D_s\|_{E\circlearrowleft}\ \le\ \|D_{m_0^{k+1}}\|_{E\circlearrowleft}\ \le\ m_0^{(k+1)\beta}\ \le\ m_0^{\beta}\,s^{\beta}.
\]
If $0<s\le 1$ (hence $k+1\le 0$), then similarly
\[
\|D_s\|_{E\circlearrowleft}\ \le\ \|D_{m_0^{k+1}}\|_{E\circlearrowleft}\ \le\ m_0^{(k+1)\alpha}\ \le\ m_0^{\alpha}\,s^{\alpha}.
\]
From the definition of the Boyd indices these bounds imply
$\alpha_E\ge \alpha>0$ and $\beta_E\le \beta<1$. Hence $\alpha_E>0$ and $\beta_E<1$, as claimed.
\end{proof}

\section*{Acknowledgements}

{We are grateful to E.~Kissin and V.~Shulman for interest to our work and their comments, in particular for drawing our attention to the fact that 
our main theorem provides a complete solution to their Problem~3.9 stated in 
\cite{KissinShulman2005}.}

\end{document}